\documentclass[11pt]{amsart}
\usepackage{amsmath, amssymb, amscd, mathrsfs, slashed, enumerate, url}
\usepackage{color}
\usepackage{xcolor}
\usepackage[all,cmtip]{xy}
\usepackage{multicol}
\usepackage{indentfirst}
\usepackage{latexsym}
\usepackage{bm}
\usepackage{graphicx}
\usepackage{subfigure,esint}
\usepackage{float}
\usepackage{verbatim}
\usepackage{comment}
\usepackage[top=1in, bottom=1in, left=1.25in, right=1.25in]{geometry}
\usepackage{epsfig,dsfont,amsthm,amsfonts,amsbsy}
\usepackage[colorlinks]{hyperref}

\newcommand{\MGC}{\mathcal{G}_0^{\mathbb{C}}}

\newcommand{\CP}{\mathbb{C}P}

\newcommand{\lan}{\langle }
\newcommand{\ran}{\rangle}

\newcommand{\ML}{\mathcal{L}}

\newcommand{\da}{\dagger}

\newcommand{\MM}{\mathcal{M}}

\newcommand{\Vol}{\mathrm{Vol}}

\newtheorem{theorem}{Theorem}[section]

\newtheorem{question}[theorem]{Question}

\newtheorem{assumption}[theorem]{Assumption}

\newtheorem{corollary}[theorem]{Corollary}

\newtheorem{definition}[theorem]{Definition}
\newtheorem{example}[theorem]{Example}

\newtheorem{lemma}[theorem]{Lemma}

\newtheorem{proposition}[theorem]{Proposition}
\newtheorem*{remark}{Remark}

\renewcommand{\Im}{\mathrm{Im}}

\newcommand{\Tr}{\mathrm{Tr}}

\newcommand{\st}{\star}
\newcommand{\we}{\wedge}

\newcommand{\ti}{\times}

\newcommand{\al}{\alpha}

\newcommand{\na}{\nabla}
\newcommand{\ep}{\epsilon}

\newcommand{\be}{\beta}

\newcommand{\mfg}{\mathfrak{g}}

\newcommand{\lam}{\lambda}

\newcommand{\mft}{\mathfrak{t}}

\newcommand{\diag}{\mathrm{diag}}

\newcommand{\MG}{\mathcal{G}}

\newcommand{\slf}{\mathfrak{sl}}

\newcommand{\Lam}{\Lambda}

\newcommand{\mg}{\mathfrak{g}}

\newcommand{\MMOd}{\MO^{\mathrm{de}}}

\newcommand{\Id}{\mathrm{Id}}

\newcommand{\mfgc}{\mathfrak{g}^{\mathbb{C}}}
\newcommand{\GC}{G^{\mathbb{C}}}
\newcommand{\hF}{\hat{F}}
\newcommand{\BR}{\mathbb{R}}
\newcommand{\BC}{\mathbb{C}}
\newcommand{\BH}{\mathbb{H}}
\newcommand{\BO}{\mathbb{O}}
\newcommand{\MN}{\mathcal{N}}

\newcommand{\MH}{\MM_{\BH}}
\newcommand{\MO}{\MM_{\BO}}

\newcommand{\NH}{\MN_{\BH}}
\newcommand{\NO}{\MN_{\BO}}

\newcommand{\mbX}{\mathbf{X}}
\newcommand{\UH}{\mathcal{H}}

\newcommand{\MMdp}{\mathcal{M}^{\mathrm{de}}_{\mathrm{pole}}}
\newcommand{\MMnc}{\mathcal{M}_{\mathrm{NC}}}

\newcommand{\MGCC}{\MG^{\mathbb{C}}}

\usepackage[utf8]{inputenc}
\usepackage[english]{babel}
\usepackage{fancyhdr}
\begin{document}
	
	\title[On the Moduli Space of the Octonionic Nahm's Equations]{On the Moduli Space of the Octonionic Nahm's Equations}
	\author{Siqi He} 
	
	\address{Simons Center for Geometry and Physics, Stony Brook University\\Stony Brook, NY, 11790}
	\email{she@scgp.stonybrook.edu}
	\maketitle
	\begin{abstract}
		In this paper, we study some basic properties of the octonionic Nahm's equations over $[0,1]$. We prove that the moduli space of the smooth solutions to the octonionic Nahm's equations over $[0,1]$ is a star-shaped smooth manifold with a complete metric. In addition, for any commuting triples of the cotangent bundle of a complex Lie group, we construct solutions to the octonionic Nahm's equations. Moreover, we introduce extra symmetry and study a decoupled version of the octonionic Nahm's equations over $[0,1]$. We prove a Kempf-Ness theorem for the meromorphic solutions to the decoupled octonionic Nahm's equations.
	\end{abstract}
	\begin{section}{Introduction}
		The classical Nahm's equations, introduced by Nahm \cite{nahm1980simple}, have been deeply studied from a number of different points of view. Among their many applications we can mention, Hitchin's construction \cite{hitchin1983construction} of monopoles, spectral curve and the Nahm transform, Kronheimer \cite{kronheimer2004hyperkahler} and Dancer \cite{dancer1993nahm}'s study on the hyperkh\"aler metric, Donaldson's \cite{donaldson1984nahmclassification} description of the moduli space of monopoles in terms of rational maps.
		
		This paper is an attempt to study an octonionic generalization of the Nahm's equations and to generalize the previous results to the new equations. Let $G$ be a compact Lie group with Lie algebra $\mg$, let $\mbX=(X_1,X_2,\cdots,X_7)$ be seven $\mathfrak{g}$-valued functions of real variable $t$ and let $\na_t$ be a covariant derivative with variable $t$, then the octonionic Nahm's equations is an ODE system which can be written as 
		$$\na_t\mbX+\mbX\times \mbX=0,$$
		where $\times$ is the cross product of vectors in $\mathbb{R}^7.$ This equation have been introduced in \cite{grabowski1993octonion} as a dimensional reduction of the $Spin(7)$ instanton equations. 
		
		Our main motivation to study the octonionic Nahm's equations is to understand the program proposed by Cherkis \cite{Cherkis2015}, which predicts an potential octonionic version of the Nahm transform. The classical Nahm transfrom theory, we refer to \cite{jardim2004surveynahm} for an overview, is closely related to the quaternions and the $4^{\mathrm{th}}$ Clifford algebra. The main difficulties come from the non associativity of the octonion algebra, which fails the Dirac operator trick that been used by Nakajima \cite{nakajima1993monopoles} and there is no twistor description. 
		
		In this paper, we study smooth solutions to the octonionic Nahm's equations analogy to the work of Dancer and Swann \cite{dancer1996hyperkahler} and prove the following:
		\begin{theorem}
			The smooth solutions to the octonionic Nahm's equations over $[0,1]$ with the $L^2$ metric is a smooth complete Riemannian manifold. In addition, it is diffeomorphic to an star shaped open set in $G\times (\mathfrak{g}\otimes \mathbb{R}^7)$. 
		\end{theorem}
		
		One of the most interesting figure in classical Nahm's equations theory is that the classical Nahm's equations can be regarded as a hyperk\"ahler moment map and $L^2$ metric becomes a hyperk\"ahler metric in the moduli space. The octonionic Nahm's equations can also be regarded as the moment map of seven complex structures which form the $7^{\mathrm{th}}$ Clifford algebra. However, neither of these complex structures can be reduced to the moduli space. We don't know whether there exists an complex structure compatible with the $L^2$ metric over the octonionic Nahm's moduli space. We also note that similar moment map description is obtained in \cite{nagy2019complex}.

		As the moduli space of the octonionic Nahm's equations might not have enough interesting structures comparing to the classical theory, we introduce extra symmetry. We write $\frac{d}{dt}+\al=\na_t+iX_1,\;\be_1=X_2+iX_3,\;\be_2=X_4+iX_5,\;\be_3=X_6+iX_7$ and assume that $[\be_i,\be_j]=0$, then the octonionic Nahm's equations can be written as 
		\begin{equation*}
		\begin{split}
		&\frac{d\be_i}{dt}+2[\al,\be_i]=0,\;[\be_i,\be_j]=0,\;for\;any\;i,j=1,2,3,\\
		&\frac{d}{dt}(\al+\al^{\st})+2([\al,\al^{\st}]+[\be_1,\be_1^{\st}]+[\be_2,\be_2^{\st}]+[\be_3,\be_3^{\st}])=0,
		\end{split}
		\end{equation*}
		which we call the decoupled octonionic Nahm's equations. 
		
		The motivation to introduce the extra symmetry comes from the relationship between the Spin(7) instanton over Calabi-Yau 4-fold and the Hermitian-Yang-Mills equations. The Hermitian-Yang-Mills equations are a special case of the Spin(7) instanton equations which have a better understood moduli space. The work of Donaldson-Uhlenbeck-Yau \cite{donaldson1985anti,uhlenbeck1986existence} fully understood the moduli space space of Hermitian-Yang-Mills in terms of algebraic data. However, the moduli space of Spin(7) instanton is still mysterious. As the decoupled octonionic Nahm's equations will be a dimension reduction of the Hermitian-Yang-Mills equations over 1-dimensional, we expect a better understanding comparing to the full octonionic Nahm's equations.
		
		Let $G$ be a compact Lie group with complexification $G^{\mathbb{C}}$, we identified the cotangent bundle $T^{\st}G^{\mathbb{C}}=G^{\mathbb{C}}\times \mathfrak{g}^{\mathbb{C}}$, then the commuting triple of $T^{\st}G^{\mathbb{C}}$ are defined as 
		$$N:=\{(g,\mft_1,\mft_2,\mft_3)\in G^{\mathbb{C}}\ti\mfg^{\mathbb{C}}\ti\mfg^{\mathbb{C}}\ti\mfg^{\mathbb{C}}|[\mft_i,\mft_j]=0\}.$$
		
		We prove a Kempf-Ness type theorem to the octonionic Nahm's equations, which generalize a theorem of Kronheimer \cite{kronheimer2004hyperkahler},
		\begin{theorem}
			There exists a one-to-one correspondence between the solutions to the decoupled octonionic Nahm's equations and the commuting triple $N$ of $T^{\st}G^{\mathbb{C}}$.
		\end{theorem}
		
		In \cite{hitchin1983construction}, Hitchin identified the moduli space of finite energy monopoles over $\mathbb{R}^3$ with solutions to the classical Nahm's equations over $(0,1)$ which is meromorphic at $0,1$ and satisfy extra symmetric condition. Donaldson \cite{donaldson1984nahmclassification} directly identified the meromorphic solutions to the classical Nahm's equations with based rational maps. 
		
		In this paper, we generalized the work of Donaldson \cite{donaldson1984nahmclassification} for $G^{\mathbb{C}}=GL(k,\mathbb{C})$. We introduce a meromorphic condition for the decoupled octonionic Nahm's equations and identified them with algebraic data. To be more explicitly, we prove
		
		\begin{theorem}
			There exists a one-to-one correspondence between
			\begin{itemize}
				\item [(a)] equivalent class of meromorphic solutions to the octonionic Nahm's equations satisfy Assumption \ref{assump},
				\item [(b)] equivalent class under $O(k,\mathbb{C})$ quadruple $(B_1,B_2,B_3,w)$ where 
				\begin{itemize}
					\item [(i)]$B_1,B_2,B_3$ are all $k\ti k$ symmetric matrix,
					\item [(ii)] a vector $w$ in $\mathbb{C}^k$ which generates $\mathbb{C}^k$ as a $\mathbb{C}[B_i]$ module for each $i=1,2,3$,
					\item [(iii)] $B_i,w$ generates the same filtration of $\mathbb{C}^k$. To be more precisely, $$\mathrm{span}\{w,B_iw,\cdots,B_i^lw\}=\mathrm{span}\{w,B_jw,\cdots,B_j^lw\}$$ for any $i,j=1,2,3$ and $l=1,\cdots,k-1$. 
				\end{itemize}   
			\end{itemize} 	
		\end{theorem}
		
		\textbf{Acknowledgements.} The author wishes to thank Simon Donaldson, \'Akos Nagy, Rafe Mazzeo, Gon\c{c}alo Oliveira, Simon Salamon, Mark Stern, Sergy Cherkis for numerous helpful discussions. 
	\end{section}
	
	\begin{section}{Cross Product and Nahm's Equations}
		In this section, we will introduce the relationship of the cross product and the Nahm's equations. 
		\begin{subsection}{Cross product and the group $G_2$}
			To begin with, we briefly introduce some linear algebra related to the octonions and $G_2$. For a nice review of octonionic related linear algebra, we refer to \cite{salamonwalpuski2017}. 
			\begin{definition}
				Let $V$ be a finite dimensional real vector space with a norm $|\cdot|$, we say $V$ is a normed division algebra if it has a structure of algebra with identity such that $|a\cdot b|=|a||b|$.
			\end{definition}
			We have the following classification theorem by Hurwitz
			\begin{theorem}{\cite{hurwitz1922}}
				there exists $4$ normed division algebra up to isomorphism, which is the real numbers $\BR$, the complex numbers $\BC$, the quaternions $\BH$ and the octonions $\BO$(non-associative).
			\end{theorem}
			For every normed division algebra $V$, we define the real part $\mathrm{Re}(V):=\mathrm{span}\{\Id\}$ and imaginary part $\Im(V):=\mathrm{Re}(V)^{\perp}$. For a normed division algebra $V$, we can define the a linear map $\times:\Im(V)\times\Im (V)\to \Im(V),\;
			a\times b:=\Im(ab),$ which we called the cross product. 
			
			We explicitly write down the cross product over $\BH$ and $\BO$ that will be used in our paper 
			\begin{example}
				When $V=\mathbb{H}$, we identified $\Im(V)$ with $\mathbb{R}^3$, then the crossproduct for vectors $\mathbf{u}=(u_1,u_2,u_3)$ and $\mathbf{v}=(v_1,v_2,v_3)$ is given by
				$$\mathbf{u}\times\mathbf{v}=\begin{pmatrix}
				u_2v_3-u_3v_2\\
				u_3v_1-u_1v_3\\
				u_1v_2-u_2v_1
				\end{pmatrix}.$$
				
				When $V=\BO$, we identified $\Im(V)$ with $\mathbb{R}^7$, then the crossproduct for vector $\mathbf{u}=(u_1,\cdots,u_7)$ and $\mathbf{v}=(v_1,\cdots,v_7)$ is given by
				
				$$\mathbf{u}\times\mathbf{v}=\begin{pmatrix}
				u_2v_3-u_3v_2+u_4v_5-u_5v_4+u_6v_7-u_7v_6\\
				-u_1v_3+v_1u_3+u_4v_6-u_6v_4-u_5v_7+u_7v_5\\
				u_1v_2-v_1u_2-u_4v_7+u_7v_4-u_5v_6+u_6v_5\\
				-u_1v_5+u_5v_1-u_2v_6+u_6v_2+u_3v_7-u_7v_3\\
				u_1v_4-u_4v_1+u_2v_7-u_7v_2+u_3v_6-v_3u_6\\
				-u_1v_7+v_1u_7+u_2v_4-u_4v_2-u_3v_5+v_3u_5\\
				u_1v_6-u_6v_1-u_2v_5+v_2u_5-u_3v_4+u_4v_3
				\end{pmatrix}.$$
			\end{example}
			
			Over $\BO$, we define the associative 3-form $\phi\in \Lambda^3(T^{\st}\Im(\BO))$ as $\phi(a,b,c):=\lan a\times b,c\ran$. If we denote $e_1,\cdots,e_7$ be an orthogonal frame of $T^{\st}\Im(\BO)$ with $e^{ijk}=e_i\we e_j\we e_k$, then in this paper, we write the 3-form as
			\begin{equation}
			\phi=\frac{1}{6}\sum_{i,j,k=1}^7f_{ijk}e^{ijk}=e^{123}+e^{145}+e^{167}+e^{246}-e^{257}-e^{347}-e^{356}.
			\end{equation} 
			
			We call $f_{ijk}$ the structure constants of the octonion and we can describe the octonion algebra as $\BO:=\{e_0,e_1,\cdots,e_7|e_0=\Id,\;e_l^{2}=-1,\;e_ie_j=f_{ijk}e_k\},$
			where $l=0,1,\cdots,7$ and $i,j=1,\cdots,7$.
			
			We define the group $G_2:=\{A\in GL(7,\mathbb{R})|A^{\st}\phi=\phi\}$ be the group that preserve the 3-form, by the work of Bryant, we have
			\begin{proposition}{\cite{bryant1987}}
				$G_2=\{A\in GL(7,\mathbb{R})||A^{\st}a|=|a|,\;A^{\st}\phi=\phi,\;A^{\st}(a \times b)=A^{\st}a\times A^{\st}b,\;\det(A)=1\}$,
				where $a,b\in \Im(\BO)$.  
			\end{proposition}
			
			As the group $G_2$ acts on $\Im\;\BO$, the induced action on the 21 dimension space $\Lam^2(\Im(\BO))$ which splits as a sum of the 7-dimensional and 14-dimensional irreducible representations that can be written as
			\begin{equation}
			\begin{split}
			\Lam^2_7:&=\{\al\in \Lam^2|\st(\phi\we\al)=2\al\}=\{\iota_{e_i}\phi\},\\
			\Lam^2_{14}:&=\{\al\in \Lam^2|\st(\phi\we\al)=-\al\}.
			\end{split}
			\end{equation}
			
			\begin{definition}
				Let $W$ be any 3-dimensional linear subspace of $\Im(\BO)$, we define $W$ be associative subspace if $\phi|_{W}=\Vol_W$.  
			\end{definition}
			
		\end{subsection}
		
		\begin{subsection}{Nahm's equations and cross product}
			Let $G$ be a compact Lie group with Lie algebra $\mg$. Let $I$ be the closed interval $[0,1]$ with coordinate $t$. We consider a trivial principle bundle $G$ bundle over $I$. Let $\na_t$ be a connection over $I$, as we are working on a trivial bundle, we write $\na_t=\frac{d}{dt}+X_0$, where $X_0:I\to \mg$ and $\frac{d}{dt}$ is the product connection over the trivial bundle.
			
			Let $V$ be a normed division algebra, we write $\mbX=(X_1,\cdots,X_k)$ where each $X_i$ are $\mg$ valued functions on $I$, $X_i:I\to \Im(V)\otimes \mg$, where $k=\dim(V)-1$. The Nahm's equations are equations of connection $\na_t$ and $\mbX$:
			\begin{equation}
			\begin{split}
			\na_t\mbX+\mbX\times \mbX=0,
			\end{split}
			\end{equation}
			where $\times: \Im(V)\otimes \mg\times \Im(V)\otimes \mg\to \Im(V)\otimes \mg$ is combination of the cross product in $\Im(V)$ and the multiplication on the Lie algebra $\mg$.
			
			We will explicitly write down the equations over $\BC,\BH,\BO$.
			When $V=\BC,\;X=X_1$, the complex Nahm's equations can be written as $\frac{d}{dt}X_1+[X_0,X_1]=0$.
			
			When $V=\BH,\;\mbX=(X_1,X_2,X_3)$, then the quaternion Nahm's equations can be written as
			\begin{equation}
			\begin{split}
			&\frac{dX_1}{dt}+[X_0,X_1]+[X_2,X_3]=0,\\
			&\frac{dX_2}{dt}+[X_0,X_2]+[X_3,X_1]=0,\\
			&\frac{dX_3}{dt}+[X_0,X_3]+[X_1,X_2]=0.
			\label{Eq_QuaternionNahm}
			\end{split}
			\end{equation}
			
			When $V=\BO,\;\mbX=(X_1,\cdots,X_7)$, then the octonionic Nahm's equations can be written as
			\begin{equation}
			\begin{split}
			&\frac{dX_1}{dt}+[X_0,X_1]+[X_2,X_3]+[X_4,X_5]+[X_6,X_7]=0,\\
			&\frac{dX_2}{dt}+[X_0,X_2]-[X_1,X_3]+[X_4,X_6]-[X_5,X_7]=0,\\
			&\frac{dX_3}{dt}+[X_0,X_3]+[X_1,X_2]-[X_4,X_7]-[X_5,X_6]=0,\\
			&\frac{dX_4}{dt}+[X_0,X_4]-[X_1,X_5]-[X_2,X_6]+[X_3,X_7]=0,\\
			&\frac{dX_5}{dt}+[X_0,X_5]+[X_1,X_4]+[X_2,X_7]+[X_3,X_6]=0,\\
			&\frac{dX_6}{dt}+[X_0,X_6]-[X_1,X_7]+[X_2,X_4]-[X_3,X_5]=0,\\
			&\frac{dX_7}{dt}+[X_0,X_7]+[X_1,X_6]-[X_2,X_5]-[X_3,X_4]=0.
			\label{Eq_octonionNahm}
			\end{split}
			\end{equation}
			Using the structure constants $f_{ijk}$, we can rewrite the octonionic equations as $$\frac{dX_i}{dt}+[X_0,X_i]+\sum_{j,k}\frac{1}{2}f_{ijk}[X_j,X_k]=0.$$
			
			The octonionic Nahm's equations have a $G_2$ action, for $A\in G_2$, as $A$ preserve the product structure, for any $(\na_t,\mbX)$ a solution to \eqref{Eq_octonionNahm}, $(\na_t,A\mbX)$ is also a solution to \eqref{Eq_octonionNahm}.
			
			We define $V:=\{I\to \mg\otimes \mathbb{R}^8\}$
			and write 
			\begin{equation*}
			\NO=\{(X_0,X_1,\cdots,X_7)\in V|\eqref{Eq_octonionNahm}\;holds\}.
			\end{equation*}
			
			Let $\MG$ be the gauge group of the trivial bundle, which is $C^2$ maps $[0,1]\to G$. Then $\MG$ acts on $V$ by
			\begin{equation}
			\begin{split}
			X_0\to gX_0g^{-1}-\frac{dg}{dt}g^{-1},\;X_i\to gX_ig^{-1}.
			\end{split}
			\end{equation}
			It is straight forward to check that equations $\eqref{Eq_octonionNahm}$ are invariant under $\MG$ action. We denote $\MG_0$ be the normal subgroup of $\MG$ defined as $\MG_0:=\{g\in\MG|g(0)=g(1)=1\}$, and we define octonionic moduli space $\MO$ be the quotient space 
			$$\MO:=\NO/\MG_0.$$
			
			Similarly, we can define $\NH$ and $\MH$ for the solutions to \eqref{Eq_QuaternionNahm}.
			
			We can easily regard solutions to the quaternion Nahm's equations' as a subsolution to the octonionic Nahm's equations moduli space. Let $(\frac{d}{dt}+Y_0,Y_1,Y_2,Y_3)$ be a solution to the quaternion Nahm's equations. We have the following observation:
			
			\begin{lemma}
				\label{Lemma_embedding}
				If we set $X_0=Y_0, X_1=Y_1,\;X_2=Y_2,\;X_3=Y_3,\;X_4=X_5=X_6=X_7$=0, then $(\frac{d}{dt}+X_0,X_1,\cdots,X_7)$ is a solution to the octonionic Nahm's equations \eqref{Eq_octonionNahm}.
			\end{lemma}
			
			\begin{proposition}
				Let $E$ be any associative plane over $\mathbb{R}^7$ which across the origin, let $\iota_E$ be the involution along $E$ on $\mathbb{R}^7$, then there exists an embedding $f_{E}:\MH\to \MO$ such that $f_{E}(\MH)$ equals to the fixed point of $\Id_{\mathbb{R}^7}-\iota_E$.
			\end{proposition}
			\begin{proof}
				By \cite{harvey1982calibrated}, the group $G_2$ is transitive over the associative planes. We denote $E_0$ be the associative plane span by $e_1,e_2,e_3$, then for the associative plane $E$, there exists $A\in G_2$ such that $AE=E_0$. The proposition follows immediately from the previous Lemma.
			\end{proof}
			
		\end{subsection}
	\end{section}
	
	\begin{section}{The Moduli space of the octonionic Nahm's equations}
		In this section, we will discuss the structures of the moduli space to the octonionic Nahm's equations.
		\begin{subsection}{The Moduli Space as a Smooth Open Manifold}
			We now are going to prove the moduli space is smooth with the method used by Kronheimer \cite[Page 3-5]{kronheimer2004hyperkahler}. Let $V=\{I\to\mg\otimes \mathbb{R}^8\}$, we have
			\begin{lemma}{\cite[Page 3-5]{kronheimer2004hyperkahler}}
				\label{lem_slicing}
				Let $(X_0,X_1,\cdots,X_7)\in V$, for each $(b_0,\cdots,b_7)\in V$, there exists $g\in\MG_0$ such that for
				$$(a_0,\cdots,a_7):=g(X_0+b_0,X_1+b_1,\cdots,X_7+b_7)-(X_0,\cdots,X_7),$$
				we have
				$\frac{da_0}{dt}+\sum_{i=0}^7[X_i,a_i]=0$. 
			\end{lemma}
			\begin{proof}
				We write $g=\exp(u)$ with $u(0)=u(1)=0$, then we define the operator $$Du=([u,A_0]-\frac{du}{dt},[u,X_1],\cdots,[u,X_7]),$$ which is the linearization of the $\MG_0$ action on $X$. We write $a=(a_0,\cdots,a_7)$, then w.r.t. the $L^2$ inner product, the formal adjoint operator $D^{\st}$ can be written as $D^{\st}a=\frac{da_0}{dt}+\sum_{i=0}^7[X_i,a_i]$. The equation $D^{\st}a=0$ can be written as $D^{\st}(Du+[b,u])=D^{\st}b$. The operator $D^{\st}D$ is invertible with the boundary condition $u(0)=u(1)=0$, then by inverse function theorem, for $b$ small, there exists a unique $g\in\MG_0$ such that $D^{\st}(g(A+b)-A)=0$.  
			\end{proof}
			For $Y_1,Y_2\in T^{\st}V$, the $L^2$ metric is defined to be 
			\begin{equation*}
			g(Y_1,Y_2)=\int_{I}\lan Y_1,Y_2 \ran,
			\end{equation*}
			where $\lan\;,\;\ran$ is an invariant inner product of $\mathfrak{g}$.
			
			\begin{proposition}
				The moduli space of the octonionic Nahm's equations with an $L^2$ metric is a smooth manifold with a complete metric.
			\end{proposition}
			\begin{proof}
				We will first shows that $\NO$ is a smooth Banach submanifold of $V$. The left-hand side of \eqref{Eq_octonionNahm} define a smooth map $\mu:V\to \Gamma$, where $\Gamma$ is the space of $C^0$ maps $I\to \mg\otimes \mathbb{R}^7$. 
				
				We consider the linearization $d\mu$ at each $(X_0,X_1,\cdots,X_7)$, for every $(\gamma_1,\cdots,\gamma_7)\in V$, we would like find $(a_0,a_1,\cdots,a_7)$ to solve the following equations for $i=1,\cdots,7$:
				\begin{equation}
				\begin{split}
				\frac{da_i}{dt}+[a_0,X_i]+[X_0,a_i]+\frac{f_{ijk}}{2}([X_j,a_k]+[a_j,X_k])=\gamma_i.
				\end{split}
				\end{equation}
				These equations have a unique solution satisfies $a_0
				\equiv 0$ and $a_i(0)=0$ for $i=1,\cdots,7$. Therefore, $d\mu$ has a right inverse, by implicit function theorem, $\NO$ is a smooth Banach submanifold.
				
				In addition, by the slicing Lemma \ref{lem_slicing}, we conclude that $\NO/\MG_0$ is a smooth manifold. In addition, the tangent space at a point $(X_0,X_1,\cdots,X_7)$ can be identified with the following equations
				\begin{equation}
				\begin{split}
				&\frac{da_0}{dt}+\sum_{i=0}^7[X_i,a_i]=0,\\
				&\frac{da_i}{dt}+[a_0,X_i]+[X_0,a_i]+\frac{f_{ijk}}{2}([X_j,a_k]+[a_j,X_k])=0,
				\end{split}
				\end{equation}
				which are linear ODE systems with $8\dim \mg$ dimensional family of solutions. 
				
				For any sequence of smooth solutions to the octonionic Nahm's equations \eqref{Eq_octonionNahm} with bounded $L^2$ metric, we can always assume $X_0=0$, then \eqref{Eq_octonionNahm} will control all the $C^k$ norms of these solutions. The completeness of the $L^2$ metric follows directly by Arzela-Ascoli theorem. 
			\end{proof}
			
			For each solution $(X_0,X=(X_1,\cdots,X_7))$ to \eqref{Eq_octonionNahm}, we can pick up an unique gauge transform $g\in \MG$ such that $X_0=-g^{-1}\frac{dg}{dt}$ with $g(0)=1$, then the octonionic Nahm's equations become the following reduced version 
			\begin{equation}
			\frac{dX}{dt}+X\times X=0.
			\label{eq_decoupledNahm}
			\end{equation}
			
			By the uniqueness of solutions to ODE with given initial value, the map $$(X_0,X)\to (g(1),X_1(0),\cdots,X_7(0)),$$ with $g$ satifies $g(0)=1$ and 
			$X_0=g^{-1}\frac{dg}{dt}$ is well-defined. In addition, it is straight forward to check that the above map is invariant under the $\MG_0$ orbit of $(X_0,X)$. Therefore, we obtain a map
			\begin{equation}
			\begin{split}
			\chi&:\MO\to G\times (\mg \otimes \mathbb{R}^7),\\
			\chi&[(X_0,X)]\to (g(1),X_1(0),\cdots,X_7(0)).
			\end{split}
			\end{equation}
			
			As the reduced equation is an ODE with smooth coefficients, given any initial value $(X_1(0),\cdots,X_7(0))$, we suppose to find an unique solution near zero. However, the initial value doesn't guarantee that the solution exists over $[0,1]$, instead, the solutions can blow-up. We have the following example
			\begin{example}
				\label{ex_nonsmooth}
				Let $G=SU(2)$, let $\sigma_1,\sigma_2,\sigma_3\in\mg$ satisfies the cyclic conditions $[\sigma_1,\sigma_2]=\sigma_3$, $[\sigma_2,\sigma_3]=\sigma_1$ and $[\sigma_3,\sigma_1]=\sigma_2$. The $(0,\frac{1}{t-1}\sigma_1,\frac{1}{t-1}\sigma_2,\frac{1}{t-1}\sigma_3,0,0,0,0)$ is a solution to the octonionic Nahm's equations, which will blows up at $t=1$.  
			\end{example}
			The following proposition characterize basic properties of $\chi$, which generalized the result of Dancer and Swann \cite{dancer1996hyperkahler} for the quaternion Nahm's equations.
			\begin{proposition}
				$\chi$ satisfies the following properties:
				\begin{itemize}
					\item [i)]$\chi$ is a diffeomorphism of $\MO$ into an open set $W$ in $G\times (\mg\otimes \mathbb{R}^7)$ containing $G\times (0,\cdots, 0)$,
					\item [ii)] $\chi$ will not be surjective unless $G$ is abelian,
					\item [iii)] the image of $\chi$ is star shaped.
				\end{itemize}
			\end{proposition}
			\begin{proof}
				For i), we first prove $\chi$ is injective. Suppose $\chi(X_0,X)=\chi(Y_0,Y)$, then there exists $g$ and $g'$ such that for $i=1,\cdots,7$, $gX_i(0)g^{-1}=g'Y_i(0)(g')^{-1}$ with $g(0)=g'(0)=1$ and $g(1)=g'(1)$. As $gX_ig^{-1}$ and $g'Y_i(g')^{-1}$ both satisfy the ODE with the same initial value, we obtain $gX_ig^{-1}=g'Y_i(g')^{-1}$ over $[0,1]$. We set $g'':=g^{-1}g'$, then $g''\in\MG_0$ and $g''Y_i(g'')^{-1}=X_i$, which implies $\chi$ is injective. 
				
				If $(\xi_1,\cdots,\xi_7)$ is close to the origin in $\mg \otimes \mathbb{R}^7$, then we can find solutions $(X_1,\cdots,X_7)$ to the reduced Nahm's equations \eqref{eq_decoupledNahm} with $X_i(0)=\xi_i$ for $i=1,\cdots,7$ and $X_i$ doesn't blows up over $[0,1]$. For every $l\in G$, we find $g\in \MG$ with $g(0)=1$ and $g(1)=l$, then we have $$\chi(-\frac{dg}{dt}g^{-1},gX_ig^{-1})=(g(1),g^{-1}X_1g|_{t=0},\cdots,g^{-1}X_7g|_{t=0})=(l,\xi_1,\cdots,\xi_7).$$ 
				
				Therefore, the image of $\chi$ contains a neighborhood of $G\times (0,\cdots,0)$. In addition, as the condition that the reduced octonionic Nahm's equations have a smooth solution over $[0,1]$ with initial value $(\xi_1,\cdots,\xi_7)$ is an open condition, we obtain that $W$ is open. $\chi$ is obviously a smooth map, we finished i). 
				
				For ii), when $G$ is non-abelian, $G$ contains a copy of $\mathfrak{su}(2)$. By Example \ref{ex_nonsmooth}, we obtain that $\chi$ is not surjective. When $G$ is abelian, all solutions to \eqref{eq_decoupledNahm} are constants in $\mg\otimes \mathbb{R}^7$, so $\chi$ is surjective. 
				
				For iii), let $X(t)=(X_1(t),\cdots,X_7(t))$ be a smooth solution to the reduced octonionic Nahm's equations \eqref{eq_decoupledNahm} over $[0,1]$ with initial value $(X_1(0),\cdots,X_7(0))=(\xi_1,\cdots,\xi_7)$. For any $0<\ep\leq 1$, $(\ep X_1(\ep t), \ep X_2(\ep t),\cdots, \ep X_7(\ep t))$ is also a solution to the reduced octonionic Nahm's equations \eqref{eq_decoupledNahm} which is smooth over $[0,\ep^{-1}]$ with initial value $(\ep\xi_1,\cdots, \ep\xi_7)$. Therefore, the image of $\chi$ is star shaped. 
			\end{proof}
		\end{subsection}
	\end{section}
	
	\begin{section}{octonionic Nahm's equations as zero sets of moment maps}
		In this section, we will introduce the moment map description of the octonionic Nahm's equations, which first appears in \cite{grabowski1993octonion}.
		\begin{subsection}{Complex Structures generated by octonion multiplications}
			As a normed vector space, we identified $\mathbb{R}^8$ with $\BO$. We choose a orthogonal base $e_0,e_1,\cdots,e_7$ of $\mathbb{R}^8$ and regard $e_0$ is the identity element in $\BO$ and $e_1,\cdots,e_7\in \Im(\BO)$. To simplify the notation, we choose a trivialization of $T^{\st}\mathbb{R}^8$ and identified each fiber of $T^{\st}\mathbb{R}^8$ with $\mathbb{R}^8$.
			
			For each $x\in \mathbb{R}^8$, we can write $x=\sum_{i=0}^7x_ie_i$, we can study the complex structure generated by octonions. Even the multiplication of octonions are non-associtive, for any $v\in\Im(\BO)$ with $v^2=-1$, we have $v(vx)=(vv)x=-x$. Therefore, for $i=1,\cdots,7$, the multiplication of $e_i$ will generates complex structure $I_i$ over $\mathbb{R}^8$. 
			
			For $V=\{I\to\mathbb{R}^8\otimes \mathfrak{g}\},$ the complex structures $I_1,\cdots,I_7$ over $\mathbb{R}^8$ will also generates complex structures over $V$. Given $Y=(Y_0,\cdots,Y_7)\in T^{\st}V$, we can explicitly write down these complex structures as
			\begin{equation}
			\begin{split}
			I_1(Y)&=(-Y_1,Y_0,-Y_3,Y_2,-Y_5,Y_4,-Y_7,Y_6),\\
			I_2(Y)&=(-Y_2,Y_3,Y_0,-Y_1,-Y_6,Y_7,Y_4,-Y_5),\\
			I_3(Y)&=(-Y_3,-Y_2,Y_1,Y_0,Y_7,Y_6,-Y_5,-Y_4),\\
			I_4(Y)&=(-Y_4,Y_5,Y_6,-Y_7,Y_0,-Y_1,-Y_2,Y_3),\\
			I_5(Y)&=(-Y_5,-Y_4,-Y_7,-Y_6,Y_1,Y_0,Y_3,Y_2),\\
			I_6(Y)&=(-Y_6,Y_7,-Y_4,Y_5,Y_2,-Y_3,Y_0,-Y_1),\\
			I_7(Y)&=(-Y_7,-Y_6,Y_5,Y_4,-Y_3,-Y_2,Y_1,Y_0).
			\end{split}
			\end{equation}
			
			Note that the octonions are non-associative, but the complex structures over tangent space are associative. The algebra generated by $I_1,\cdots,I_7$ will be different from the octonion algebra. We can explicitly compute the relationships between these complex structures. 
			
			For each index $(ijk)$ such that the construct constant $f_{ijk}\neq 0$, we define a map 
			\begin{equation*}
			\begin{split}
			\iota_{0ijk}&:T^{\st}V\to T^{\st}V\\
			\iota_{0ijk}&(Y_0,\cdots,Y_i,\cdots,Y_j,\cdots,Y_k,\cdots,Y_7)=(-Y_0,\cdots,-Y_i,\cdots,-Y_j,\cdots,-Y_k,\cdots,Y_7), 
			\end{split}
			\end{equation*}
			which changes the sign in front of $Y_0,Y_i,Y_j,Y_k$.
			
			\begin{lemma}
				These complex structures satisfies the following relationships:
				\begin{itemize}
					\item [i)]For any $i\neq j$, $I_i\circ I_j=-I_j\circ I_i$.
					\item [ii)]For any $i\neq j\neq k$, $f_{ijk}I_i\circ I_j\circ I_k=\iota_{0ijk}$. To be more explicitly, we have 
					\begin{equation*}
					\begin{split}
					&I_1 \circ I_2 \circ I_3=\iota_{0123},\;I_1 \circ I_4 \circ I_5=\iota_{0145},\;I_1 \circ I_6 \circ I_7=\iota_{0167},\\
					&I_2 \circ I_4 \circ I_6=\iota_{0246},-I_2 \circ I_5 \circ I_7=\iota_{0257},\;-I_3 \circ I_4 \circ I_7=\iota_{0347}.\;-I_3 \circ I_5 \circ I_6=\iota_{0356}.
					\end{split}
					\end{equation*}
				\end{itemize}
			\end{lemma}
			\begin{proof}
				This is proved by straight forward computation.
			\end{proof}
			The following corollary is well-known
			\begin{corollary}
				$I_1,\cdots,I_7$ generate the $7^{\mathrm{th}}$ Clifford algebra in $\mathrm{End}(T^{\st}V)$.
			\end{corollary}
			
			We can also write down the symplectic structures corresponding to these complex structures. Recall that for $Y_1,Y_2\in T^{\st}V$, the $L^2$ metric will defined to be 
			\begin{equation*}
			g(Y_1,Y_2)=\int_{I}\lan Y_1,Y_2 \ran,
			\end{equation*}
			where $\lan\;,\;\ran$ is a metric of $\mathfrak{g}$.
			
			For $i=1,\cdots,7$, we denote $\al_i:=e^{0i}+\iota_{e_i}\phi$, which we explicitly write as
			\begin{equation}
			\begin{split}
			&\al_1=e^{01}+e^{23}+e^{45}+e^{67},\;
			\al_2=e^{02}-e^{13}+e^{46}-e^{57},\;
			\al_3=e^{03}+e^{12}-e^{47}-e^{56},\\
			&\al_4=e^{04}+e^{51}+e^{62}-e^{73},\;
			\al_5=e^{05}+e^{14}-e^{72}-e^{63},\;
			\al_6=e^{06}+e^{71}+e^{24}-e^{35},\\
			&\al_7=e^{07}+e^{16}-e^{25}-e^{34}.
			\end{split}
			\end{equation}
			
			Note that the forms $\iota_{e_i}\phi$ will be a base of $\Lambda^2_7(\mathbb{R}^7)$, where $\mathbb{R}^7$ is the $7$ dimensional plane span by $e_1,\cdots,e_7$. We can define $7$ symplectic forms $\omega_i$
			\begin{equation}
			\omega_i(Y_1,Y_2):=\int_I \al_i(Y_1,Y_2),
			\label{eq_symplecticstructures}
			\end{equation}
			and it is straight forward to check that $\omega_i(Y_1,Y_2)=g(I_iY_1,Y_2)$.

		\end{subsection}
		
		\begin{subsection}{octonionic Nahm's equations as a moment map}
			For any $X=(X_0,\cdots,X_7)\in V$, recall the gauge group $g\in\MG_0$ acts on $X$ by $$X_0\to -\frac{dg}{dt}g^{-1}+gX_0g^{-1},\;X_i\to gX_ig^{-1},\;for\;i\neq 0.$$ Then the $\MG_0$ will be Hamiltonion for every symplectic structure $\omega_i$ defined in \eqref{eq_symplecticstructures}. We write $\xi:I\to\mg$ be an element in the Lie algebra of $\MG_0$. We denote $$X^{\xi}:=\frac{d}{dt}|_{t=0}\exp(t\xi) X=(-\frac{d}{dt}\xi,-[X_1,\xi],\cdots,-[X_7,\xi]).$$
			
			We denote $\mu_i$ be the moment map to the symplectic structure $\omega_i$. For $Y\in T_X^{\st}V$, we denote $(I_iY)_k$ be the value of the $k$-th coordinate of $I_iY$, then we compute 
			\begin{equation}
			\begin{split}
			d\lan \mu,\xi\ran(Y)&=-\omega_i(X_{\xi},Y)=-g(X_{\xi},I_iY)\\
			&=\int_0^1\lan\frac{d}{dt}\xi+[X_0,\xi],(I_iY)_0\ran+\sum_{k=1}^7\lan[X_k,\xi],(I_iY)_k\ran\\
			&=\int_{0}^1\lan \xi,(-\frac{d}{dt}(I_iY)_0-[X_0,(I_iY)_0]-\sum_{k=1}^7[X_k,(I_iY)_k]).
			\end{split}
			\end{equation}
			In addition, $(I_iY)_0=-Y_i$, $(I_iY)_k=f_{ijk}Y_j$ for $k\neq i$ and $(I_iY)_i=Y_0$, we obtain that 
			\begin{equation}
			\begin{split}
			d\mu_i(Y)&=\frac{d}{dt}Y_k-[X,I_kY]\\
			&=\frac{d}{dt}Y_i+[X_0,Y_i]+[Y_0,X_i]+\frac{1}{2}\sum_{j,k=1}^7f_{ijk}([X_j,Y_k]+[Y_j,X_k]).
			\label{eq_tangentspaceequationsofNO}
			\end{split}
			\end{equation}
			Therefore, the moment map will be
			\begin{equation}
			\mu_i(X)=\frac{d}{dt}X_i+[X_0,X_i]+\frac{1}{2}\sum_{j,k=1}^7f_{ijk}[X_j,X_k],
			\end{equation}
			which is the $i$-th equations in \eqref{Eq_octonionNahm}. 
			Therefore, we obtain 
			
			\begin{proposition}
				$\NO=\cap_{i=1}^7\mu_i^{-1}(0)$ and $\MO=\cap_{i=1}^7\mu_i^{-1}(0)/\MG_0$.
			\end{proposition}
			
			It will be interesting to know that what kind of structure will the moduli space holds, but at least the natural complex structure comes from the moment maps can not be reduced to the moduli space.
			
			The tangent space equations of $\MO$ at $X$ can be written as $\frac{d}{dt}Y_k-[X,I_kY]=0$, where $[X,I_kY]:=\sum_{i=0}^7[X_i,(I_kY)_i]$. 
			We obtain $d\mu_{I_i}(I_jY)=-\frac{d}{dt}Y_k+[X,\iota_{0ijk}I_kY]$. Therefore, for any $Y\in T^{\st}_X\MO$ , it must satisfies 
			$[X,\iota_{0ijk}I_k+I_kY]=0.$ for any $(ijk)$ with $f_{ijk}\neq 0$.

			\begin{proposition}
				Suppose $G$ is non-abelian, then for any $i=1,\cdots,7$, $I_i$ doesn't preserve the tangent space of $\NO$ and can not be reduced to an almost complex structure of $\MO$.
			\end{proposition}
			\begin{proof}
				For $X\in \NO$, for every $Y\in T_x\NO$, $Y$ will satisfies \eqref{eq_tangentspaceequationsofNO}, which are 
				\begin{equation}
				\begin{split}
				d\mu_k(Y)=\frac{d}{dt}Y_{k}-[X,I_kY]=0.
				\label{eq_tangentspace}
				\end{split}
				\end{equation}
				Suppose $I_j$ preserve the complex structures for $j\neq k$, then $d\mu_k(I_jY)=0$ for any $k=1,\cdots,7$. 
				
				By a straight forward computation, the equations $d\mu_k(I_jY)=0$ are equivalent to $$-\frac{d}{dt} Y_k+[X,\iota_{0ijk}I_kY]=0,$$
				where $(ijk)$ are index with $f_{ijk}\neq 0$. 
				
				Therefore, suppose $I_k$ preserve the tangent space of $\NO$, for any $X\in \NO$ and $Y\in T_X\NO$, they must satisfies $[X,(\iota_{0ijk}\circ I_k+I_k)Y]=0$ for any $(ijk)$ span an associate plane. For any $l$ different from $0,i,j,k$, we pick up a trivial solution of \eqref{Eq_octonionNahm} with $X_l(0)=\xi\in\mg$ and $X_{l'}(0)=0$ for $l'\neq l$. Then $[X_l,(I_kY)_l]=0$, However, as the tangent space equations \eqref{eq_tangentspace} for $Y_0\equiv 0$ will be an ODE, where we can always find solutions will an given initial values. Suppose $G$ is non-abelian, we can always find $[X_l(0),(I_kY)_l(0)]\neq 0$, which gives a contradiction.
				
				As the complex structures doesn't preserve the tangent space of $\NO$, they can not be reduced to $\MO$.
			\end{proof}
			
		\end{subsection}
	\end{section}
	
	\begin{section}{Constructing Solutions from Commuting Triples}
		In this section, we will introduce a decoupled version for the octonionic Nahm's equations and we will prove a Kempf-Ness type theorem for the decoupled octonionic Nahm's equations.  
		
		\begin{subsection}{The Decoupled octonionic Nahm's equations}
			Let $G$ be a compact Lie group and $G^{\mathbb{C}}$ be a complexified of $G$. Let $\mg_{\mathbb{C}}=\mg\oplus i\mg$ be the complexified Lie algebra. 
			
			We write $\al=\frac{1}{2}(X_0+iX_1)$, $\be_1=\frac{1}{2}(X_2+iX_3)$, $\be_2=\frac{1}{2}(X_4+iX_5)$, $\be_3=\frac{1}{2}(X_6+iX_7)$ and define the following operators
			\begin{equation*}
			\begin{split}
			&d_{\al}=\frac{1}{2}\frac{d}{dt}+\al,\;d_{\be_1}=[\be_1,\;],\;d_{\be_2}=[\be_2,\;],\;d_{\be_3}=[\be_3,\;],\\
			&\bar{d}_{\al}=\frac{1}{2}\frac{d}{dt}-\al^{\st},\;\bar{d}_{\be_1}:=-[\be_1^{\st},\;],\;\bar{d}_{\be_2}=-[\be_2^{\st},\;],\;\bar{d}_{\be_3}=-[\be_3^{\st},\;].
			\end{split}
			\end{equation*}
			The octonionic Nahm's equations can be re-written as 
			\begin{equation}
			\begin{split}
			&[d_{\al},d_{\be_1}]+[\bar{d}_{\be_2},\bar{d}_{\be_3}]=0,\;[d_{\al},d_{\be_3}]+[\bar{d}_{\be_1},\bar{d}_{\be_2}]=0,\;[d_{\al},d_{\be_3}]+[\bar{d}_{\be_1},\bar{d}_{\be_2}]=0,\\
			& [d_{\al},\bar{d}_{\al}]+[d_{\be_1},\bar{d}_{\be_1}]+[d_{\be_2},\bar{d}_{\be_2}]+[d_{\be_3},\bar{d}_{\be_3}]=0.
			\label{Eq_CY4form}
			\end{split}
			\end{equation}
			
			We introduce extra condition that $[\be_i,\be_j]=0$ for any $i,j=1,2,3$, then we call the following equation the decoupled octonionic Nahm's equations:
			\begin{equation}
			\begin{split}
			&\frac{d\be_i}{dt}+2[\al,\be_i]=0,\;[\be_i,\be_j]=0,\;for\;any\;i,j=1,2,3,\\
			&\hF(\al,\be):=\frac{d}{dt}(\al+\al^{\st})+2([\al,\al^{\st}]+[\be_1,\be_1^{\st}]+[\be_2,\be_2^{\st}]+[\be_3,\be_3^{\st}])=0.
			\end{split}
			\label{eq_decoupledNahm}
			\end{equation}
			
			To motivate our treatment of these equations, we have
			\begin{proposition}
				We consider $\mathbb{C}^4$ with coordinate $(z_0=t+iy_1,z_1,z_2,z_3)$ and the canonical K\"ahler structure, then the system \eqref{eq_decoupledNahm} is equivalent to the Hermitan-Yang-Mills equations over $\mathbb{C}^4$ which are invariant in the direction of $y_1,z_1,z_2,z_3$.
			\end{proposition}
			
			In addition, the equations \eqref{eq_decoupledNahm} have extra $SU(3)$ symmetry. Let $(\al,\be_1,\be_2,\be_3)$ is a solution and let $A\in SU(3)$, then $A$ could acts on $(\be_1,\be_2,\be_3)$ by multiplication and $(\al,A(\be_1,\be_2,\be_3))$ is still a solution.
			
			Similarly, we define $\MMOd$ be the moduli space of the decoupled octonionic Nahm's equations \eqref{eq_decoupledNahm}:
			$$\MMOd:=\{(\al,\be_1,\be_2,\be_3)\;\mathrm{smooth\;solutions\;to\;\eqref{eq_decoupledNahm}\;over\;[0,1]}\}/\MG_0.$$ 
			
			However, the decoupled octonionic Nahm's equations have extra symmetry. We define $\MGC:=\{g:[0,1]\to G^{\mathbb{C}}\}$, then the action of $g$ is defined as \begin{equation}
			\begin{split}
			\al'=g(\al)&=g\al g^{-1}-\frac{1}{2}\frac{dg}{dt}g^{-1},\\
			\be'_i=g(\be_i)&=g\be_ig^{-1}.
			\end{split}
			\end{equation} 
			This action preserve the first system of \eqref{eq_decoupledNahm}. To be more precisely, we have 
			\begin{equation*}
			\begin{split}
			\frac{d\be'_i}{dt}+2[\al',\be_i']=g(\frac{d\be_i}{dt}+2[\al,\be_i])g^{-1},\;[\be_i',\be_j']=g^{-1}[\be_i,\be_j]g.
			\end{split}
			\end{equation*} 
			
			In addition, the second system $\hF(\al,\be_1,\be_2,\be_3)=0$ could be regarded as the moment map equation for the $\MG_0$ action. We call the first line of \eqref{eq_decoupledNahm} the complex equation and the second line of \eqref{eq_decoupledNahm} real equation.
			
			For $(\al,\be_1,\be_2,\be_3)$, we can choose a complex gauge $g\in\MGC$ such that $\al'=g\al g^{-1}-\frac{1}{2}\frac{dg}{dt}g^{-1}=0$. We write $\be_i'=g\be_ig^{-1}$, then the first system of \eqref{eq_decoupledNahm} become $\frac{d}{dt}\be_i'=0$, which means $g \be_i g^{-1}$ are independent of $t$. We denote $N$ be the commuting triple over $T^{\st}\GC$:
			\begin{equation}
			\begin{split}
			N:=\{(g,t_1,t_2,t_3)\in\GC\ti \mfgc\ti\mfgc\ti\mfgc|[t_i,t_j]=0\},
			\end{split}
			\end{equation}
			
			and the previous argument defines the following map by forgetting the second system of \eqref{eq_decoupledNahm}:
			\begin{equation}
			\begin{split}
			\Theta:&\MMOd \to N,\\
			&(\al,\be_1,\be_2,\be_3)\to (g(1),g\be_2g^{-1},g\be_3g^{-1},g\be_4g^{-1}),
			\label{eq_KHmap}
			\end{split}
			\end{equation}
			where $g\in\MGCC$ is the unique element solves $\frac{dg}{dt}=g\al$ with $g(0)=\Id$.
			
			In the rest of this section, we will prove
			\begin{theorem}
				\label{thm_kobayashihitchin}
				The map $\Theta$ defined in \eqref{eq_KHmap} is a bijection.
			\end{theorem}
		\end{subsection}

		\begin{subsection}{The Existence and Uniqueness of Solutions}
			We first prove $\Theta$ is surjective with variational method. As the form of the decoupled octonionic Nahm's equations \eqref{eq_decoupledNahm} is very closed to the classical ones, our existence arguments will largely use the method introduced by Donaldson \cite{donaldson1984nahmclassification}
			\begin{lemma} 
				For $g\in \MGCC$, then for any $(\al',\be_1',\be_2',\be_3')=g(\al,\be_1,\be_2,\be_3)$, we have
				\begin{equation}
				\begin{split}
				g^{-1}\hF(\al',\be_1',\be_2',\be_3')g-\hF(\al,\be_1,\be_2,\be_3)=-2\bar{d}_{\al}(h^{-1}d_{\al}h)-2\sum_{i=1}^3\bar{d}_{\be_i}(h^{-1}d_{\be_i} h),
				\label{eq_indentityuniquenesslemma}
				\end{split}
				\end{equation}
				where $h=g^{\st}g$.
			\end{lemma}
			\begin{proof}
				Note that we can write $\hF(\al,\be_1,\be_2,\be_3)=2([d_{\al},\bar{d}_{\al}]+\sum_{i=1}^3[d_{\be_i},\bar{d}_{\be_i}])$. The Lemma follows directly from the following identities
				\begin{equation*}
				\begin{split}
				&g^{-1}\circ\bar{d}_{\al'}\circ g=\bar{d}_{\al},\;
				g^{-1}\circ d_{\al}\circ g=d_{\al}+h^{-1}d_{\al}h,\\
				&g^{-1}\circ\bar{d}_{\be_i}\circ g=\bar{d}_{\be_i},\;g^{-1}\circ d_{\be_i}\circ g=d_{\be_i}+h^{-1}d_{\be_i}h.
				\end{split}
				\end{equation*}
			\end{proof}
			
			\begin{proposition}
				\label{prop_unique}
				Given $(\al,\be_1,\be_2,\be_3)\in V$ and $g\in\MGCC$, we write $(\al',\be_1',\be_2',\be_3')=g(\al,\be_1,\be_2,\be_3)$ and define a functional of $g$:
				\begin{equation}
				\ML(g)=\int_0^1|\al'+(\al')^{\st}|^2+2\sum_{i=1}^3|\be'_i|^2.
				\end{equation}
				Then $\hF(\al',\be_1',\be_2',\be_3')=0$ if and only if $\ML$ is stationary w.r.t. variation $\delta g$ supported in $(0,1)$.
			\end{proposition}
			\begin{proof}
				We compute the variation $\delta g$. Without loss of generality, we suppose $g=1$, so $\al=\al',\;\be=\be'$ and $\delta g$ is self-adjoint. Then $\delta \al=[\delta g,\al]-\frac{1}{2}(\delta g)$ and $\delta \beta_i=[\delta g,\be_i]$. $\delta(\al+\al^{\da})=[\delta g,\al-\al^{\da}]-\frac{d}{dt}(\delta g)$.
				
				We compute the variation $\delta\ML$:
				\begin{equation}
				\begin{split}
				\delta \ML&=\int \Tr((\al+\al^{\st})\delta(\al+\al^{\st})+2\sum_{i=1}^3(\be_i(\delta \be_i)^{\st})dt\\
				&=\int \Tr(-(\al+\al^{\st})\frac{d}{dt}\delta g+(\al+\al^{\st})[\delta g,\al-\al^{\st}]+2\sum_{i=1}^3(\be_i[\be_i^{\st},\delta g])\\
				&=\int \Tr(\delta g(\frac{d}{dt}(\al+\al^{\st})+2([\al,\al^{\st}]+\sum_{i=1}^3[\be_i,\be_i^{\st}]))\\
				&=\int \Tr(\delta g\hF(\al,\be_1,\be_2,\be_3)).
				\end{split}
				\end{equation}
			\end{proof}
			
			\begin{proposition}
				\label{prop_existence}
				For $(\al,\be_1,\be_2,\be_3)$ satisfies the complex equation over $[0,1]$, then for any $h_-,h_+$, there is a continous $h:[0,1]\to GL(k,\mathbb{C})$ such that
				\begin{itemize}
					\item [i)]$F(\al',\be')=0$,
					\item [ii)]$h(0)=h_-$, $h(1)=h_+.$
				\end{itemize}
			\end{proposition}
			\begin{proof}
				Let $\al=0$, $\be_1,\be_2,\be_3$ are constant elements in the Lie algebra. We write $h=g^{\st}g$, then the Lagrangian integrand is 
				\begin{equation}
				\begin{split}
				&|\frac{dg}{dt}g^{-1}+(g^{\st})^{-1}\frac{dg^{\st}}{dt}|^2+\sum_{i=1}^3|g\be_ig^{-1}|^2\\
				=&\Tr(h^{-1}\frac{dh}{dt})^2+\sum_{i=1}^3\Tr(\be_ih^{-1}\be_i^{\st}h).
				\end{split}
				\end{equation}
				Let $\mathcal{H}=G^{\mathbb{C}}/G$ be the homogeneous space, then $\Tr(h^{-1}\frac{dh}{dt})^2$ is the standard $G^{\mathbb{C}}$ invariant Riemannian metric on $\mathcal{H}$. So this gauge the Lagrangian $\ML=\int \frac 12 |\frac{dh}{dt}|^2_{\mathcal{H}}+V(h)$, with $V(h)=\sum_{i=1}^3\Tr(\be_ih^{-1}\be_i^{\st}h)\geq 0$. 
				
				From method of calculus of variations that for any $h_-,h_+\in\mathcal{H}$, there exits a path $h:[0,1]\to \mathcal{H}$ smooth in (0,1), minimizing Lagrangian $\ML$ and satisfies the boundary condition $h(0)=h_-$, $h(1)=h_+$. For any $g$ satisfies $g^{\st}g=h$, we set $(\al',\be_1',\be_2',\be_3')=g(\al,\be_1,\be_2,\be_3)$, then $\hF(\al',\be_1',\be_2',\be_3')=0$.
				$$$$
				
				For $h\in\UH$, we define a function $\sigma$ measures the eigenvalues of $h$:
				$$\sigma(h):=\Tr(h)+\Tr(h^{-1})-2.$$ Suppose $\lam_1,\cdots,\lam_k$ be eigenvalues of $h$, then $\sigma(h)=\sum_{i=1}^k(\lam_i+\lam_i^{-1}-2)\geq 0$. In addition, when $\sigma(h)=0$, then $h\equiv 0$. 
				\begin{lemma}
					\label{lem_unique}
					The function $\sigma$ will satisfy 
					\begin{equation}
					\frac{d^2}{dt^2}\sigma\geq -(|\hF|+|\hF'|).
					\end{equation}
				\end{lemma}
				\begin{proof}
					We multiply $h$ in both side of \ref{eq_indentityuniquenesslemma}, then we obtain
					\begin{equation}
					\begin{split}
					\bar{d}_{\al}d_{\al}h=&\frac{1}{2}(h\hF-g^{\st}\hF'g)+\bar{d}_{\al}h h^{-1}\bar{d}_{\al}h+\sum_{i=1}^3\bar{d}_{\be_i}h h^{-1} d_{\be_i}h-\sum_{i=1}^3\bar{d}_{\be_i}d_{\be_i}h.
					\label{eq_idk}
					\end{split}
					\end{equation}
					Note that
					\begin{equation*}
					\begin{split}
					\Tr(\bar{d}_{\al}d_{\al}h)=\frac{d^2}{dt^2}\Tr(h),\;\Tr(\bar{d}_{\be_i}d_{\be_i}h)=0,\\
					\Tr(\bar{d}_{\al}h h^{-1}d_{\al}h)\geq 0,\;\Tr(\bar{d}_{\be_i}h h^{-1}d_{\be_i}h)\geq 0,
					\end{split}
					\end{equation*}
					We take trace in both side of \eqref{eq_idk} and obtain $\frac{d^2}{dt^2}\Tr(h)\geq -\frac{1}{2}(|\hF|+|\hF'|).$ Take $h$ to $h^{-1}$, we obtain  $\frac{d^2}{dt^2}\Tr(h^{-1})\geq -\frac{1}{2}(|\hF|+|\hF'|)$. The proposition follows by adding them together.
				\end{proof}
				
				This inequality implies the uniqueness of the solution.
				\begin{proposition}
					\label{prop_uniqueness}
					Let $(\al,\be_1,\be_2,\be_3)$ be a solution to the complex equations of \eqref{eq_decoupledNahm} in $(0,1)$, suppose for $g_1,g_2\in \MGCC$ with $(\al',\be_1',\be_2',\be_3'):=g_1(\al,\be_1,\be_2,\be_3),\;(\al'',\be_1'',\be_2'',\be_3''):=g_2(\al,\be_1,\be_2,\be_3)$, they satisfy the equations $\hF(\al',\be_1',\be_2',\be_3')=\hF(\al'',\be_1'',\be_2'',\be_3'')=0$. In addition, let $h_1:=g_1^{\st}g_1$ and $h_2:=g_2^{\st}g_2$, suppose $h_1(0)=h_2(0)$ and $h_1(1)=h_2(1)$, then $g_1^{\st}g_1\equiv g_2^{\st}g_2$ over $[0,1]$.
				\end{proposition}
				\begin{proof}
					WLOG, we assume $g_2=1$ and $(\al'',\be_1'',\be_2'',\be_3'')=(\al,\be_1,\be_2,\be_3)$. We write $h_1=g_1^{\st}g_1$, then by Lemma \eqref{lem_unique}, we obtain $\frac{d^2}{dt^2}\sigma(h_1)\geq 0$. In addition, as $\sigma(h_1)\geq 0$ and $\sigma(h_1)(0)=\sigma(h_1)(1)=0$, we obtain $h_1\equiv 1$.
				\end{proof}
				
				\rm{Proof of Theorem \ref{thm_kobayashihitchin}:} $\Theta$ is injective follows by Proposition \ref{prop_unique}. For any $(g_0,t_1,t_2,t_3)\in N$, we define $\al=0$ and $\be_1=t_1,\be_2=t_2,\be_3=t_3$. By proposition \ref{prop_existence}, we choose the boundary condition $h_-=1$ and $h_+=g_0^{\st}g_0$, we obtain an solution and this proves $\Theta$ is surjective.
			\end{proof}
			
		\end{subsection}
	\end{section}
	
	\begin{section}{octonionic Nahm's equations with Poles}
		In this section, we will study the reduced octonionic Nahm's equations over $(0,1)$ which is meromorphic at $t=0$ and $t=1$. The study of the meromorphic boundary condition starts in Hitchin \cite{hitchin1983construction} and later by Donaldson \cite{donaldson1984nahmclassification}. In this section, we generalize their results to the decoupled Nahm's equationscase and we only considered the case when $G^{\mathbb{C}}=GL(k,\mathbb{C})$.
		\begin{subsection}{The simple pole at $t=0$}
			To begin with, we will introduce the pole boundary condition. Let $A$ be an $k\ti k$ matrix, we write $A^{\intercal}$ be the transpose of the matrix and $A^{\st}$ be the conjugation transpose of the matrix. Recall that an principle $\mathfrak{sl}_2$ triple on $GL(k,\mathbb{C})$ can be written as
			\begin{equation}
			\begin{split}
			&a_0=\mathrm{diag}(-\frac{k-1}{4},\;-\frac{k-1}{4}+\frac{1}{2},\cdots,\frac{k-1}{4}),\\
			&b_0=\begin{pmatrix}
			0 & 0& \cdots&  &0\\
			\gamma_1&0 &\cdots & &0\\
			0& \gamma_2&\cdots & &0\\
			\cdots &\cdots & \cdots& & \cdots\\
			0& & & \gamma_{k-1}& 0
			\end{pmatrix},
			\end{split}
			\end{equation}
			where $\gamma_i=\frac{\sqrt{i(k-i)}}{2}$. Then $a_0,b_0,b_0^{\st}$ satisfy $$
			b_0=2[a_0,b_0],-a+[b_0,b_0^{\st}]=0.$$
			
			\begin{assumption}
				\label{assump}
				Let $(\al=X_0+iX_1,\be_1=X_2+iX_3,\be_2=X_4+iX_5,\be_3=X_6+iX_7)$ be a solutions to \eqref{eq_decoupledNahm} over $(0,1)$, we make the following assumptions:
				\begin{itemize}
					\item [(i)] for $i=0,\cdots,7$, $X_i^{\st}(t)=-X_i(t)$, $X_i(1-t)=X_i(t)^{\intercal}$
					\item [(ii)] for integer $1\leq i\leq 7$, $X_i$ extends to a meromoprhic function on a neighborhood of $[0,1]$ with simple poles at $t=0,1$. $X_0$ is continuous over $[0,1],$
					\item [(iii)] The real gauge group $\MG$$(\mathrm{similarly,\;} \MGCC)$ consists of continuous map $g:[0,1]\to U(k)$ $(g:[0,1]\to GL(k,\mathbb{C}))$ such that $g$ is smooth in $(0,1)$ and $g(1-t)=g^{\intercal}(t)^{-1}$. 
					\item [(iv)] let $a,b_1,b_2,b_3$ be the residue of $\al,\be_1,\be_2,\be_3$, then there exists three complex numbers $s_1,s_2,s_3$ with $|s_1|^2+|s_2|^2+|s_3|^2=1$ such that after suitable $\MG$ action, we can write $a=a_0$ and $b_i=s_ib_0$, and $(a_0,b_0,b_0^{\st})$ is a principle $\slf_2$ triple on $GL(k,\mathbb{C})$.
				\end{itemize}
			\end{assumption}
			
			We define the following moduli space
			$$\MMdp:=\{\mathrm{Solutions\;to\;\eqref{eq_decoupledNahm}\;with\;Assumption\;\ref{assump}}\}/\MG.$$
			The following remark explain that why we introduction $\mathrm{(iv)}$ in Assumption \ref{assump}.
			\begin{remark}
				By $\mathrm{(i),(ii),(iii)}$ of the Assumption \ref{assump}, near $t=0$, we can write
				$$\al=\frac{a}{t}+\mathcal{O}(1),\;a^{\st}=a,\;\be_i=\frac{b_i}{t}+\mathcal{O}(1),\mathrm{for\;i=1,2,3},
				$$
				then $t^{-2}$ order terms of \eqref{eq_decoupledNahm} becomes 
				\begin{equation}
				\begin{split}
				&b_1=2[a,b_1],\;b_2=2[a,b_2],\;b_3=2[a,b_3],\;[b_i,b_j]=0,\\
				&-a+[b_1,b_1^{\st}]+[b_2,b_2^{\st}]+[b_3,b_3^{\st}]=0.
				\label{eq_poleconditiont=0}
				\end{split}
				\end{equation}
				Note that the first system of the above equation is invariant by a constant multiplication on $b_i$. Suppose $av=\lam v$, where $v$ is an eigenvector and $\lam$ is the corresponding eigenvalue, then $a(b_iv)=(\lam+\frac{1}{2})b_iv$. 
				
				In addition, suppose the action of $b_1$ on an eigenvector $v$ generates $\mathbb{C}^k$ and $\Tr(a)=0$, we much have $av=(k-\frac{1}{4})v$ and under suitable gauge, we can write $a=a_0,\;b_1=b_0$. 
				We also requires that $b_2$ and $b_3$ acts on $v$ generates $\mathbb{C}^k$, by the condition $[b_i,b_j]=0$, we conclude that $b_2=\mu_2b_0$, $b_3=\mu_3b_0$, where $\mu_2,\mu_3$ are two complex numbers. Moreover, in order to satisfy the real equations, we need to unify the constants in front of $b_0$. Let $C=\sqrt{1+|\mu_2|^2+|\mu_3|^2}$, we write $g_C=\diag(1,C,\cdots,C^{k-1})$, then $g_Cb_0g_C^{-1}=Cb_0$. Therefore, we can set $s_1=\frac{1}{C}, s_2=\frac{\mu_2}{C}$ and $s_3=\frac{\mu_3}{C}$ and the leading term of the real equation of  \eqref{eq_poleconditiont=0} becomes 
				$$-a_0+(|s_1|^2+|s_2|^2+|s_3|^2)[b_0,b_0^{\st}],$$
				which vanishes as $a_0,b_0,b_0^{\st}$ are a principle $\slf_2$ triple.
			\end{remark}
			
			\begin{definition}
				A Nahm complex consist of $(\al,\be_1,\be_2,\be_3,v)$, where $\al,\be_1,\be_2,\be_3$ are $GL(k,\mathbb{C})$-valued functions over $(0,1)$ and $v\in \mathbb{C}^k$ a unit vector $(|v|=1)$ satisfy the followings:
				\begin{itemize}
					\item [(i)] $\al,\be_1,\be_2,\be_3$ satisfy the complex equations
					\begin{equation*}
					\begin{split}
					\frac{d\be_i}{dt}+2[\al,\be_i]=0,\;[\be_i,\be_j]=0,\;\mathrm{for\;i,j=1,2,3},
					\end{split}
					\end{equation*}
					\item [(ii)] $\al(1-t)=\al^{\intercal}(s),\;\be_i(1-t)=\be_i^{\intercal}(t)$,
					\item [(iii)] $\al,\be_1,\be_2,\be_3$ are smooth in $(0,1)$, meromorphic with simple poles at $t=0,1$, and residues $a,b_1,b_2,b_3$ at $t=0$,
					\item [(iv)] $\Tr(a)=0$ and $v$ is a vector in the $-(\frac{k-1}{4})$ eigenspace of $a$ such that the vectors $\{b_i^jv\}_{j=0}^{k-1}$ span $\mathbb{C}^k$ for $i=1,2,3$.
				\end{itemize}
			\end{definition}
			In the above definition, when $\be_2=\be_3=0$, these is exactly the condition that introduced by Donaldson \cite[Definition 1.21]{donaldson1984nahmclassification}. Suppose $s_1=1$, then $s_2=0$ and $s_3=0$, this is exactly the boundary condition that has been introduced by Witten \cite{witten2011fivebranes} for the Kapustin-Witten equations and studied in \cite{HeMazzeo2017}.
			
			From the previous discussion in the remark, we obtain
			\begin{corollary}
				\label{coro_complexgaugerightasymptotic}
				For any Nahm complex $(\al,\be_1,\be_2,\be_3,v)$, there exists $g\in\MGCC$ such that $g(\al,\be_1,\be_2,\be_3)$ satisfies $\mathrm{(iv)}$ of Assumption \ref{assump} and $|g(v)|=1$ is still an unit vector.
			\end{corollary}

			Even though the asymptotic behaviors of $\be_1,\be_2,\be_3$ are the same, it doesn't means that they are complex gauge equivalent. To see this, we give the following example of the Nahm complexes and we will see later a classification of all Nahm complexes.
			
			\begin{example}
				Let $q_1,q_2,q_3$ be three different real numbers, let $B_i=\begin{pmatrix}
				q_i& 1\\
				1 & q_i+1
				\end{pmatrix}$. We choose any function $f(t)$ such that $f(1-t)=f(t)^{-1}$, $f(t)=t^{-1}$ near $t=0$ and $f(t)=1$ near $t=\frac{1}{2}$, we define the complex gauge transform $p(t)=\diag(f(t)^{\frac12},f(t)^{-\frac 12})$. 
				
				We define $\al:=\frac{1}{2}p^{-1}\frac{dp}{dt}$ and $\be_i:=p(t)^{-1}B_ip(t)$, then near $t=0$, we have 
				\begin{equation*}
				\begin{split}
				\al&=\frac{1}{4}\begin{pmatrix}
				\log f & \\
				& -\log f
				\end{pmatrix}=\frac{1}{t}\begin{pmatrix}
				-\frac{1}{4} & \\
				0 & \frac{1}{4} 
				\end{pmatrix}+\mathcal{O}(1),\\
				\be_i&=\begin{pmatrix}
				q_i & f^{-1}\\
				f & q_i+1
				\end{pmatrix}=\frac{1}{t}\begin{pmatrix}
				0 & 0\\
				1 & 0
				\end{pmatrix}+\mathcal{O}(1)
				\end{split}
				\end{equation*}
				
				In addition, as $\Tr(\be_i)=\Tr(B_i)=2q_i+1$, and $q_i\neq q_j$, it is impossible that $\be_i$ is complex gauge equivalent to $\be_j$ for any $i\neq j$.
			\end{example}
			
			\begin{definition}
				Two Nahm complexes $(\al,\be_1,\be_2,\be_3,v)$, $(\al',\be_1',\be_2',\be_3',v')$ are equivalent if there is a continuous map $g:[0,1]\to GL(k,\mathbb{C})$ smooth in the interior, such that
				\begin{itemize}
					\item [i)] $g(\al,\be_1,\be_2,\be_3)=(\al',\be_1',\be_2',\be_3')$ in $(0,1)$,
					\item [ii)] $g(1-t)=g^{\intercal}(t)^{-1}$,
					\item [iii)] $g(0)v=v'$.
				\end{itemize}
			\end{definition}
			We write $\MMnc$ be the moduli space of Nahm complexes quotient by equivalent relationship in the above definition. Similarly, we can define a map between $\MMdp$ and $\MMnc$ by forgetting the real equations of \eqref{eq_decoupledNahm} 
			\begin{equation*}
			\Xi:\MMdp\to\MMnc.
			\end{equation*}
			We first prove $\Xi$ is injective.
			\begin{proposition}
				\label{prop_poleinjective}
				Let $X=(X_0,X_1,\cdots,X_7),\;X'=(X_0',X_1',\cdots,X_7')\in\MMdp$, we write $\Xi(X)=(\al,\be_1,\be_2,\be_3,v)$ and $\Xi(X')=(\al',\be'_1,\be'_2,\be'_3,v')$. Suppose there exists $g\in \MGCC$ such that $g\Xi(X)=\Xi(X')$, then $X$ and $X'$ are unitary gauge equivalent.  
			\end{proposition}
			\begin{proof}
				By Assumption \ref{assump} $(iv)$, there exists $g_0\in \MG$ that that $|g_0(X)-X'|=\mathcal{O}(1)$. We take $\tilde{g}:=gg_0^{-1}$, then at $t=0$, we have $\tilde{g}^{-1}a_0\tilde{g}=a_0,\;\tilde{g}^{-1}b_0\tilde{g}=b_0,\;|\tilde{g}v|=1.$ Therefore, by the representation theory of $\slf_2$ at $t=0$, $\tilde{g}=1$ thus $g|_{t=0}=g_0|_{t=0}$ which is unitary. In addition, as $g(1-t)=g^{\intercal}(t)^{-1}$, we define $h:=g^{\st}g$, then $h|_{t=0}=h|_{t=1}=\Id$. By Proposition \ref{prop_uniqueness}, we obtain $h=\Id$ over $(0,1)$, which implies that $g$ is unitary. 
			\end{proof}
			
		\end{subsection}
		
		\begin{subsection}{Existence of Solutions}
			
			\begin{lemma}
				\label{lemma_goodasympto}
				Suppose $(\al,\be_1,\be_2,\be_3,v)$ be a Nahm complex, then there is an equivalent Nahm complex $(\al',\be_1',\be_2',\be_3')$ such that
				\begin{itemize}
					\item [(i)] $|\al-\al'^{\st}|$ is bounded in $[0,1]$,
					\item [(ii)] $|v'|=1$,
					\item [(iii)] The residue of $\al'$ at $t=0$ is $a_0$ and the residues of $\be_1',\be_2',\be_3'$ are $s_1b_0,s_2b_0,s_3b_0$, where $s_1,s_2,s_3\in \mathbb{C}$ with $|s_1|^2+|s_2|^2+|s_3|^2=1,$
					\item [(iv)] $\hF(\al',\be_1',\be_2',\be_3')$ is bounded in $[0,1]$.
				\end{itemize}
			\end{lemma}
			\begin{proof}
				(i), (ii), (iii) follow directly by the definition and Corollary \ref{coro_complexgaugerightasymptotic}. Near a neighborhood of $t=0$, we can write 
				$$\hF(\al,\be_1,\be_2,\be_3)=\frac{1}{s^2}(-(a_0+a_0^{\st})+2([a_0,a_0^{\st}]+[b_0,b_0^{\st}]))+\frac{1}{s}\Delta+\mathcal{O}(1).$$
				
				Now, we consider a complex gauge transform 
				$g=1+\frac{\chi}{2}t+\cdots,\;\;\mathrm{with}\;\chi=\chi^{\st},
				$ and write $(\al',\be_1',\be_2',\be_3')=g(\al,\be_1,\be_2,\be_3)$. We write $h=g^{\st}g$, then we have
				$$
				g^{-1}\hF(\al',\be_1',\be_2',\be_3')g=\hF(\al,\be_1,\be_2,\be_3)-2(\bar{d}_{\al}(h^{-1}d_{\al}h)+\sum_{i=1}^3\bar{d}_{\be_i}(h^{-1}d_{\be_i}h)).
				$$
				Let $\Delta'$ be the residue of $\hF(\al',\be_1',\be_2',\be_3')$ at $t=0$ transforms by
				$$
				\Delta'=\Delta+2([b_0,[b_0^{\st},\chi]]+[a_0,[a_0,\chi]])-[a_0,\chi].
				$$
				Let $W^{k-1}$ be the $k$ dimensional irreducible representation of $\slf_2$. We define $$P\chi:=4([b_0,[b_0^{\st},\chi]]+[a_0,[a_0,\chi]])-2[a_0,\chi],$$ then this is the Casimir operator on the representation 
				$$
				W^{k-1}\otimes W^{k-1}\cong W^{2k-2}\oplus \cdots \oplus W^0,
				$$ 
				and w.r.t the decomposition, $P|_{W^l}=\frac{l(l+2)}{2}$. Therefore, we could solve the equation $\Delta'=0$ only if when the $W^0$ component of $\Delta$ vanishes, which is $\Tr(\Delta)=0$. As the $\mathcal{O}(-1)$ terms of $\hF(\al,\be_1,\be_2,\be_3)$ all comes from the communicators, we have $\Tr(\Delta)=0$, thus, we can always solve the equations. 
			\end{proof}
			
			\begin{lemma}
				For each $0<\ep<\frac{1}{4}$, $over (\ep,1-\ep)$, if $(\al,\be_1,\be_2,\be_3)$ is a Nahm complex, there exists an unique complex gauge transform $g_{\ep}\in \MGCC$ satisfies  $g_{\ep}^{\st}g_{\ep}|_{\ep}=(g_{\ep}^{\st}g_{\ep})|_{1-\ep}=1$ and $g_{\ep}(1-t)=g_{\ep}^{\intercal}(t)$ such that
				$g_{\ep}(\al,\be_1,\be_2,\be_3)$ is a solution to \eqref{eq_decoupledNahm}.
			\end{lemma}
			\begin{proof}
				It follows directly by Proposition \ref{prop_existence} and Proposition \ref{prop_unique}.
			\end{proof}
			
			\begin{proposition}
				\label{prop_poleexistence}
				Let $(\al,\be_1,\be_2,\be_3)$ be a Nahm complex such that $|\hF(\al,\be_1,\be_2,\be_3)|\leq C$ over $[0,1]$, let $h_{\ep}=g_{\ep}^{\st}g_{\ep}$ be a solution over $[\ep,1-\ep]$ with $h_{\ep}|_{\ep}=h_{\ep}|_{1-\ep}=1$, then the following holds
				\begin{itemize}
					\item [a)] there exists a uniform constant $C$ such that $|h_{\ep}|\leq C$, $|h_{\ep}^{-1}|\leq C$ and $|h_{\ep}(t)-1|\leq Ct$.
					\item [b)] there exists continuous map $h:[0,1]\to \mathcal{H}$ such that over $(0,1)$, $h_{\ep}$ convergence smoothly to $h$. In addition, the following holds for $h$
					\begin{itemize}
						\item [i)] $h(1-t)=h^{\intercal}(s)^{-1}$, $h(0)=h(1)=1$ and $|h(t)-\Id|\leq Ct$,
						\item [ii)] for any continuous $g$ with $g^{\st}g=h$, $g(\al,\be_1,\be_2,\be_3)$ satisfies the real equation of \eqref{eq_decoupledNahm}.
					\end{itemize}   
				\end{itemize}
			\end{proposition}
			\begin{proof}
				We write $\sigma_{\ep}:=\Tr(h_{\ep})+\Tr(h_{\ep})^{-1}-2$, then by Lemma \ref{lem_unique}, we obtain 
				\begin{equation}
				\frac{d^2}{dt^2}\sigma_{\ep}\geq -2|\hF(\al,\be_1,\be_2,\be_3)|\geq -2C,
				\end{equation}
				where the last inequality is by our assumption. In addition, as $h_{\ep}|_{t=\ep,\;t=1-\ep}=1$, we obtain $\sigma_{\ep}|_{t=\ep,\;t=1-\ep}=0$. Let $\eta_{\ep}(t)=C(t-\ep)(1-\ep-t)$, then we have $\frac{d^2}{dt^2}(\eta_{\ep}-\sigma)\leq 0$ and $(\eta_{\ep}-\sigma)|_{t=\ep,1-\ep}=0$. 
				
				Therefore, by maximal principal, $$\sigma_{\ep}\leq C(t-\ep)(1-t-\ep)\leq Ct(2-t)+2.$$
				
				As the subspace $\{h|\sigma(h)\leq C\}$ is compact \cite[
				Page 93]{Donaldsonboundary}, by a diagonal argument, we obtain $h_{\ep}$ convergence to an $h$ over $[0,1]$. In addition, by the regularity of $h_{\ep}$, we conclude that this convergence is smooth over $(0,1)$.
			\end{proof}
			
			We still need to show that the behavior near the poles are as expected. We set $g=h^{\frac{1}{2}}$ and write $(\al',\be_1',\be_2',\be_3')=g(\al,\be_1,\be_2,\be_3)$, then by definition,
			$$
			\al'=g\al g^{-1}-\frac12\frac{dg}{dt}g^{-1},\;\be_1'=g\be_1g^{-1},\;\be_2'=g\be_2g^{-1},\;\be_3'=g\be_3g^{-1}.
			$$ 
			As $\lim_{t\to 0}h=\Id$, we still need to show that $\frac{dg}{dt}$ would not contribute any new poles. We using the following argument similar to Donaldson \cite[Lemma 2.20]{donaldson1984nahmclassification}.
			\begin{lemma}
				\label{prop_polerightbehavior}
				$|\frac{dg}{dt}|\leq C$ over $(0,1)$.
			\end{lemma}
			\begin{proof}
				As $g=h^{\frac12}$, it is enough to prove $|\frac{dh}{dt}|\leq C$. To save notation, the $C$ in this proof is a uniform constant might have different values in different lines. As $(\al',\be_1',\be_2',\be_3')$ is a solution, we have
				\begin{equation}
				\frac{1}{2}\hF(\al,\be_1,\be_2,\be_3)=\bar{d}_{\al}(h^{-1}d_{\al}h+\sum_{i=1}^3\bar{d}_{\be_i}h^{-1}d_{\be_i}h).
				\label{eq_temp}
				\end{equation}
				
				We define 
				\begin{equation*}
				\begin{split}
				\Omega_1:&=\frac12([\al,h^{-1}\frac{dh}{dt}]+h^{-1}\frac{dh}{dt}h^{-1}[\al^{\st},h]-h^{-1}[\frac{d\al^{\st}}{dt},h]-h^{-1}[\al^{\st},\frac{dh}{dt}]), \\
				\Omega_2:&=[\al,h^{-1}[\al^{\st},h]]+\sum_{i=1}^3[\be_1,h^{-1}[\be_1^{\st},h]]-\frac12 \hF(\al,\be_1,\be_2,\be_3).
				\end{split}
				\end{equation*}
				We expand \eqref{eq_temp} and obtain 
				\begin{equation}
				\frac{1}{4}\frac{d}{dt}(h^{-1}\frac{dh}{dt})+\Omega_1+\Omega_2=0.
				\label{eq_temp2}
				\end{equation}
				
				In addition, as $|h-1|\leq Ct$, we have $$|[\al^{\st},h]|\leq C,\;[\be^{\st},h]\leq C,\;|[\frac{d\al^{\st}}{dt},h]|\leq Ct^{-1}.$$
				Combining these estimates and Lemma \ref{lemma_goodasympto}, we have $|\al-\al^{\st}|\leq C$. Therefore, from \eqref{eq_temp2}, we obtain
				$$|\frac{d}{dt}(h^{-1}\frac{dh}{dt})|\leq C(|\frac{dh}{dt}|+t^{-1}).$$
				Given any $t_0$ and for any $t_1\in [\frac12 t_0,\frac32 t_0]$, we have 
				$$|h^{-1}\frac{dh}{dt}|_{t=t_1}-h^{-1}\frac{dh}{dt}|_{t=t_0}|\leq C(Mt_0+1),$$
				with $M:=\sup_{[\frac{1}{2}t_0,\frac{3}{2}t_0]}|\frac{dh}{dt}|.$ 
				
				As $|h-1|\leq Ct$, for $t_0\leq \frac{1}{2}M^{-1}$, we take $t_1$ to be the point that $|\frac{dh}{dt}|$ achieved $M$, we obtain 
				$|M|\leq C(|\frac{dh}{dt}|_{t=t_0}|+1).$ In addition, still using $|h-1|\leq Ct$, we obtain 
				$$|\frac{dh}{dt}|_{t=t_1}-\frac{dh}{dt}|_{t=t_0}|\leq C(Mt_0+1),$$
				
				After integrating $t_1$ along $[\frac{1}{2}t_0,\frac{3}{2}t_0]$, we obtain
				$$|\frac{h(\frac{3}{2}t_0)-h(\frac{1}{2}t_0)}{t_0}-\frac{dh}{dt}|_{t=t_0}|\leq C(Mt_0+1).$$ As $|\frac{h(\frac{3}{2}t_0)-h(\frac{1}{2}t_0)}{t_0}|\leq Ct_0$, for small $t_0$, we have $M\leq C$ which is the desire estimate. 
			\end{proof}
			
			Now, we can state our main theorem
			\begin{theorem}
				The map $\Xi$ is a bijection.
			\end{theorem}
			\begin{proof}
				It follows directly from Proposition \ref{prop_poleexistence}, \ref{prop_poleinjective} and Lemma \ref{prop_polerightbehavior}.
			\end{proof}
			
		\end{subsection}
		
		\begin{subsection}{Classification of Nahm complexes}
			In this subsection, we will give a classification of Nahm complexes in terms of weighted symmetric matrix.
			
			Let $B$ be a $k\times k$ symmetric matrix, we same $w\in \mathbb{C}^k$ is a weight of $B$ if $w$ generates $\mathbb{C}^k$ as a $\mathbb{C}[B]$ module. 
			
			Using the weight, we can define a filtration $$F_{(B,w)}^{\bullet}=\{0=F^0_{(B,w)}\subset F^1_{(B,w)}\subset \cdots F^{k}_{(B,w)}=\mathbb{C}^k\},$$ where
			$F^i_{(B,w)}:=\mathrm{span}\{w,Bw,\cdots,B^{i-1}w\}\subset \mathbb{C}^k.$
			
			We recall the following well-known lemma about asymptotic behavior of solutions $\bar{d}_{\al}s=0.$
			\begin{lemma}{\cite[Page 90, Theorem 2]{coppel1965stability}, \cite[Page 183]{hitchin1983construction}}
				\label{lemma_ODE}
				We write $\al=\frac{a_0}{t}+a_1$ with $a_0=\diag(-\frac{k-1}{4},\cdots,\frac{k-1}{4})$, then the equation $\frac{d}{dt}s+2\al s=0$ has a fundamental system of solutions $s_1(t),s_2(t),\cdots, s_k(t)$ such that $t^{-\frac{k-1}{2}+i-1}s_i(t)\to e_i$, where $e_i$ is an unit eigenvector correspondence to the eigenvalue $\frac{k-1}{2}-(i-1).$
			\end{lemma}
			
			Therefore, there exists an unique element $s(t)$ such that $\frac{d}{dt}s+2\al s=0$, $\lim_{t\to 0}t^{-\frac{k-1}{2}}s(t)=v$. We define $w=s(\frac{1}{2})$ and denote $B_i=\be_i(\frac 12)$. The quadruple $(B_1,B_2,B_3,w)$ satisfies some properties which we summarized as the follows.
			
			\begin{lemma}
				The quadruple $(B_1,B_2,B_3,w)$ satisfies the following properties
				\begin{itemize}
					\item [(i)] $B_1,B_2,B_3$ are symmetric matrix such that $[B_i,B_j]=0$ for $i,j=1,2,3$,  
					\item [(ii)] $w$ is acyclic vector for $B_1,B_2,B_3$,
					\item [(iii)] let $F_{(B_i,w)}^{\bullet}$ be the filtration defined by $B_i$ and $w$, then $F_{(B_1,w)}^{\bullet}=F_{(B_2,w)}^{\bullet}=F_{(B_3,w)}^{\bullet}$. In addition, we can write
					\begin{equation*}
					F_{(B_i,w)}^{l}=\{s|_{t=\frac{1}{2}}|\bar{d}_{\al}s=0,\;\lim_{t\to 0}t^{-\frac{k-1}{2}+l-1}s\;\mathrm{exists}\}.
					\end{equation*}
				\end{itemize}	
			\end{lemma}
			\begin{proof}
				(i) follows directly from our assumption $\be_i(t)=\be_i^{\intercal}(1-t)$ and $[\be_i,\be_j]=0$. 
				
				For integer $l$, the following space $$G_i^l:=\{s|\bar{d}_{\al}s=0,\;\lim_{t\to 0}t^{-\frac{k-1}{2}+l-1}s\;\mathrm{exists}\},$$ defines a filtration for of $\mathbb{C}^k$ for each $t$. We denote $s_0$ be the vector that $\bar{d}_{\al}s_0=0$ and $\lim_{t\to 0}t^{-\frac{k-1}{2}}s_0=v$. As $\lim_{t\to 0}tB_i=s_ib_0$ when, for each integer $1 \leq m\leq k$, the limit $\lim_{t\to 0}t^{-\frac{k-1}{2}+m}(\be_i)^ms_0=b_0^mv$ exists and nonvanishing. In addition, as $\bar{d}_{\al}(\be_i^ms_0)=0$, we obtain $\be_i^ms_0\in G_i^l$ for each $m\leq l$. Moreover, $v,b_0v,\cdots,b^{m-1}v$ are all eigenvector of $a_0$, by Lemma \ref{lemma_ODE}, we conclude that
				$$
				G_i^l=\mathrm{span}\{s_0,\be_is_0,\cdots,\be_i^{l-1}s_0\}.
				$$
				In addition, when restricting $G_i^l$ to $t=\frac12$, we conclude that $G_i^l|_{t=\frac 12}=F^l_{(B_i,w)}$. (ii) and (iii) follows directly from this description.
			\end{proof}
			
			\begin{definition}
				We call $(B_1,B_2,B_3,w)$ a Nahm quadruple if the following holds
				\begin{itemize}
					\item [(i)] $B_1,B_2,B_3$ are symmetric matrix with $[B_i,B_j]=0$ for $i,j=1,2,3$,
					\item [(ii)] $w$ generates $\mathbb{C}^k$ as a $\mathbb{C}[B_i]$ module for each $i=1,2,3$,
					\item [(iii)] $(B_i,w)$ generates the same filtration of $\mathbb{C}^k$: $F_{(B_1,w)}^{\bullet}=F_{(B_2,w)}^{\bullet}=F_{(B_3,w)}^{\bullet}.$
				\end{itemize}
			\end{definition}
			
			For $A\in O(k,\mathbb{C})$, $A$ acts on $(B_1,B_2,B_3,w)$ by $B_i\to AB_iA^{-1}$ and $w\to Aw$, then we have
			\begin{lemma}
				\label{lemma_equivalentNahmquadruple}
				$(B_1,B_2,B_3,w)$ depends only on the equivalence class of the Nahm complexes $(\al,\be_1,\be_2,\be_3)$ up to the action of $O(k,\mathbb{C})$. 
			\end{lemma} 
			\begin{proof}
				It follows directly by the definition.
			\end{proof}
			
			We define the moduli space of Nahm quadruple as 
			\begin{equation}
			\MM_{Symm}:=\{\mathrm{Nahm\;quadruple\;}(B_1,B_2,B_3,w)\}/O(k,\mathbb{C}).
			\end{equation}
			For each Nahm complex $(\al,\be_1,\be_2,\be_3,v)$, we can define a Nahm quadruple $(B_1,B_2,B_3,w)$ by taking $B_i:=\be_i|_{t=\frac{1}{2}}$ and $w=s_0|_{t=\frac12}$, where $s_0$ is the unique element satisfies $\bar{d}_{\al}s_0=0$ and $\lim_{t\to 0}t^{-\frac{k-1}{2}}s(t)=v.$ By Lemma \ref{lemma_equivalentNahmquadruple}, we obtain a well-defined map 
			$$\kappa:\MMnc\to \MM_{Symm}.$$ Now, we will construct the inverse map of $\kappa$.
			
			\begin{lemma}
				For any Nahm quadruple $(B_1,B_2,B_3,w)$, there exists a Nahm complex $(\al,\be_1,\be_2,\be_3,v)$ such that $\kappa(\al,\be_1,\be_2,\be_3,v)=(B_1,B_2,B_3,w)$. 
			\end{lemma}
			\begin{proof}
				As $\{w,B_1w,\cdots,B_1^{k-1}w\}$ span $\mathbb{C}^k$, in these bases, we write $B=GB'G^{-1}$, where $B'$ can be written as 
				\begin{equation*}
				\begin{split}
				B_1'=\begin{pmatrix}
				0 & 0&\cdots & 0&\st \\
				1& 0& \cdots& \cdots& \st\\
				0& 1& \cdots& \cdots& \st\\
				\cdots& \cdots& \cdots& 0& \st\\
				0& \cdots& 0& 1&  \st
				\end{pmatrix}.
				\end{split}
				\end{equation*}
				In addition, as $F_{(B_1,w)}^{\bullet}=F_{(B_2,w)}^{\bullet}=F_{(B_3,w)}^{\bullet},$ we can write $B_2=GB_2'G^{-1}$ and $B_3=GB_3'G^{-1}$ with
				\begin{equation*}
				\begin{split}
				B_2'=\begin{pmatrix}
				\st & \st&\cdots & \st&\st \\
				\gamma_1& \st& \cdots& \cdots& \st\\
				0& \gamma_2& \cdots& \cdots& \st\\
				\cdots& \cdots& \cdots& \st& \st\\
				0& \cdots& 0& \gamma_{k-1}&  \st
				\end{pmatrix},\;B_3'=\begin{pmatrix}
				\st & \st&\cdots & \st&\st \\
				\mu_1& \st& \cdots& \cdots& \st\\
				0& \mu_2& \cdots& \cdots& \st\\
				\cdots& \cdots& \cdots& \st& \st\\
				0& \cdots& 0& \mu_{k-1}&  \st
				\end{pmatrix}.
				\end{split}
				\end{equation*}
				We write $e_i=B_1^{i-1}w$, then $B_1'e_i=e_{i+1}$, $B_2'e_i=\gamma_{i}e_{i+1}\;(\mod\;F^{i}_{(B_1,w)})$ and $B_3'e_i=\mu_{i}e_{i+1}(\mod\;F^{i}_{(B_1,w)})$. In addition, from $[B'_1,B'_2]e_i=0$ implies $\gamma_i=\gamma_{i+1}$, similarly, $\mu_i=\mu_{i+1}$ follows by $[B_1',B_3']=0$. Therefore, we write  
				\begin{equation*}
				\begin{split}
				B_2'=\gamma\begin{pmatrix}
				\st & \st&\cdots & \st&\st \\
				1& \st& \cdots& \cdots& \st\\
				0& 1& \cdots& \cdots& \st\\
				\cdots& \cdots& \cdots& \st& \st\\
				0& \cdots& 0& 1&  \st
				\end{pmatrix},\;B_3'=\mu\begin{pmatrix}
				\st & \st&\cdots & \st&\st \\
				1& \st& \cdots& \cdots& \st\\
				0& 1& \cdots& \cdots& \st\\
				\cdots& \cdots& \cdots& \st& \st\\
				0& \cdots& 0& 1&  \st
				\end{pmatrix}.
				\end{split}
				\end{equation*}
				We choose any function $f(t)$ such that $f(1-t)=f(t)^{-1}$, $f(t)=\frac{t}{1+|\gamma|^2+|\mu|^2}$ near $t=0$ and $f(t)=1$ near $t=\frac{1}{2}$, we define the complex gauge transform $p(t)=\diag(f(t)^{\frac{k-1}{2}},\cdots,f(t)^{-\frac{k-1}{2}}).$ Then near $t=0$, we have \begin{equation*}
				\begin{split}
				p(t)B_1'p(t)^{-1}\sim\frac{1}{1+|\gamma|^2+|\mu|^2}\begin{pmatrix}
				0 & 0&\cdots & 0&\st \\
				t^{-1}& 0& \cdots& \cdots& \st\\
				0& t^{-1}& \cdots& \cdots& \st\\
				\cdots& \cdots& \cdots& 0& \st\\
				0& \cdots& 0& t^{-1}&  \st
				\end{pmatrix},\\
				p(t)B_2'p(t)^{-1}\sim\frac{\gamma}{1+|\gamma|^2+|\mu|^2}\begin{pmatrix}
				\st & \st&\cdots & \st&\st \\
				t^{-1}& \st& \cdots& \cdots& \st\\
				0& t^{-1}& \cdots& \cdots& \st\\
				\cdots& \cdots& \cdots& \st& \st\\
				0& \cdots& 0& t^{-1}&  \st
				\end{pmatrix},\\
				p(t)B_3'p(t)^{-1}\sim\frac{\mu}{1+|\gamma|^2+|\mu|^2}\begin{pmatrix}
				\st & \st&\cdots & \st&\st \\
				t^{-1}& \st& \cdots& \cdots& \st\\
				0& t^{-1}& \cdots& \cdots& \st\\
				\cdots& \cdots& \cdots& \st& \st\\
				0& \cdots& 0& t^{-1}&  \st
				\end{pmatrix},
				\end{split}
				\end{equation*}
				while $$-\frac{1}{2}p^{-1}\frac{dp}{dt}\sim\frac{1}{t}\mathrm{diag}(-\frac{k-1}{2},\cdots,\frac{k-1}{2}).$$
				Therefore, for $i=1,2,3$, $\be_i(t)=GpBg^{-1}G^{-1}$ and $\al=\frac{1}{2}G(p^{-1}\frac{dp}{dt})G^{-1}$ satisfies the complex equations of \eqref{eq_decoupledNahm}. If we take $v=Gw$, there we obtain a Nahm complex with the desire asymptotic behavior.
			\end{proof}
			
			\begin{lemma}
				Let $(\al,\be_1,\be_2,\be_3,v)$, $(\al',\be_1',\be_2',\be_3')$ be two Nahm complexes and we write $(B_1,B_2,B_3,w)$ and $(B_1',B_2',B_3',w')$ be the corresponding Nahm quadruple. Suppose there exists $G\in O(k,\mathbb{C})$ such that $GB_iG^{-1}=B_i'$, $Gw=w'$, then there exists a complex gauge transform $g\in\MGCC$ such that $g(\al,\be_1,\be_2,\be_3)=(\al',\be_1',\be_2',\be_3')$. 
			\end{lemma}
			\begin{proof}
				The equation $g\bar{d}_{\al}g^{-1}=\bar{d}_{\al'}$ for $g$ is an ODE and we can always solve with $g(\frac{1}{2})=G$ over $(0,1)$. In addition, based on our definition of $w$ where exists an unique element $s_{0}$ with $d_{\al}s_0=0$, $s_0(\frac{1}{2})=w$ and $s_0(0)=v$ and similarly, we have $s_0'$ corresponding to $d_{\al'}$ and $v'$. 
				
				For each $l=0,1,\cdots, k$, we have $$\bar{d}_{\al'}(g\be_i^ls_0)=g^{-1}\bar{d}_{\al}(\be_i^ls_0)=0,\;\bar{d}_{\al'}((\be_i')^ls_0')=0.$$ In addition, as $g\be_i^ls_0=(g\be_ig^{-1})^lgs_0$ and $g\be_i^ls_0|_{t=\frac12}=GB_i^lw=(B_i')^lw'=(\be_i')^ls_0'|_{t=\frac12}$, by uniqueness of solutions to ODE with given initial value at $t=\frac12$, we conclude that $g\be_i^ls_0=(\be_i')^ls_0'$. As $s_0$($s_0'$) is cyclic for $\be_i$($\be_i'$), we conclude $g\be_ig^{-1}=\be_i'$.
			\end{proof}
			
			Summarizing all the discussions in this section, we obtain
			\begin{corollary}
				The map $\kappa:\MMnc\to \MM_{Symm}$ is a bijection.
			\end{corollary}
		\end{subsection}
		
		\begin{subsection}{Nahm quadruple and rational maps}
			We write $R_k$ be the set of the rational map $f:\CP^1\to \CP^1$ with degree $k$, with $f(0)=\infty$. When we identified $\CP^1=\mathbb{C}\cup\{\infty\}$, then every $f\in R_k$ can be written as $f(z)=\frac{p(z)}{q(z)}$, where $p(z),q(z)$ are coprime polynomials such that $\deg\;q=k,\;\deg\;p\leq k-1$. We have the following relationship between the rational map and weighted symmetric matrix:
			\begin{proposition}{\cite[Proposition 3.1]{donaldson1984nahmclassification}}
				Let $B$ be a $k\times k$ symmetric matrix and $w$ is a cyclic vector for $B$, then the following map
				$$(B,w)\to w^{\intercal}(z\Id-B)^{-1}w$$ induces a one-to-one correspondence between $O(k,\mathbb{C})$ equivalent classes of pair $(B,w)$ and $R_k$. 
			\end{proposition}
			
			Therefore, for each Nahm quadruple $(B_1,B_2,B_3,w)$, we could associate three based point rational map $f_1,f_2,f_3$ defined as $f_i:=w^{\intercal}(z\Id-B_i)^{-1}w$. However, we don't find obvious relationship between these three based rational maps. We would like to ask the following question:
			\begin{question}
				Given any three based point rational maps $f_1,f_2,f_3$, when it comes from a Nahm quadruple?
			\end{question}
			
			Currently, we can't answer this question, but we will give a computation when $k=2$
			\begin{example}
				When $k=2$, let $(B_1,B_2,B_3,w)$ be a Nahm quadruple, we define the non-vanishing complex number $\tau:=w^{\intercal}w$, then there exists an unique $g\in O(2,\mathbb{C})$ such that $gw=\sqrt{\tau}\begin{pmatrix}
				1\\
				0
				\end{pmatrix}$, using this gauge and the condition that $[B_i,B_j]=0$, we conclude that $B_i=\begin{pmatrix}
				p_i & q_i\\
				q_i & p_i+q_is
				\end{pmatrix},$ where $p_i,s\in\mathbb{C}$ and $q_i\in \mathbb{C}^{\st}$. Then the corresponding based point rational maps can be written as
				\begin{equation*}
				f_i(z)=\frac{\tau(z-p_i-q_is)}{(z-p_i)(z-p_i-q_is)-q_i^2}.
				\end{equation*}
				Thus, when $k=2$ the space of Nahm quadruple can be parametrized by eight independent variables $(p_1,p_2,p_3,q_1,q_2,q_3,\tau,s)$. However, the space of three degree two rational maps can be parametrized by $12$ independent variables, which has much larger freedom.
			\end{example} 
		\end{subsection}
	\end{section}

	\bibliographystyle{plain}
	\bibliography{references}

\begin{thebibliography}{10}

\bibitem{bryant1987}
Robert~L. Bryant.
\newblock Metrics with exceptional holonomy.
\newblock {\em Ann. of Math. (2)}, 126(3):525--576, 1987.

\bibitem{Cherkis2015}
Sergey~A. Cherkis.
\newblock Octonions, monopoles, and knots.
\newblock {\em Lett. Math. Phys.}, 105(5):641--659, 2015.

\bibitem{coppel1965stability}
W.~A. Coppel.
\newblock {\em Stability and asymptotic behavior of differential equations}.
\newblock D. C. Heath and Co., Boston, Mass., 1965.

\bibitem{dancer1996hyperkahler}
Andrew Dancer and Andrew Swann.
\newblock Hyperk{\"a}hler metrics associated to compact lie groups.
\newblock In {\em Mathematical Proceedings of the Cambridge Philosophical
  Society}, volume 120, pages 61--69. Cambridge University Press, 1996.

\bibitem{dancer1993nahm}
Andrew~S. Dancer.
\newblock Nahm's equations and hyper-{K}\"{a}hler geometry.
\newblock {\em Comm. Math. Phys.}, 158(3):545--568, 1993.

\bibitem{donaldson1984nahmclassification}
S.~K. Donaldson.
\newblock Nahm's equations and the classification of monopoles.
\newblock {\em Comm. Math. Phys.}, 96(3):387--407, 1984.

\bibitem{donaldson1985anti}
Simon~K. Donaldson.
\newblock Anti-self-dual {Y}ang-{M}ills connections over complex algebraic
  surfaces and stable vector bundles.
\newblock {\em Proc. London Math. Soc. (3)}, 50(1):1--26, 1985.

\bibitem{Donaldsonboundary}
Simon~K. Donaldson.
\newblock Boundary value problems for {Y}ang-{M}ills fields.
\newblock {\em J. Geom. Phys.}, 8(1-4):89--122, 1992.

\bibitem{grabowski1993octonion}
Marek~P. Grabowski and Chia~Hsiung Tze.
\newblock On the octonionic {N}ahm equations and self-dual membranes in {$9$}
  dimensions.
\newblock In {\em Symmetries in science, {VI} ({B}regenz, 1992)}, pages
  287--297. Plenum, New York, 1993.

\bibitem{harvey1982calibrated}
Reese Harvey, H~Blaine Lawson, et~al.
\newblock Calibrated geometries.
\newblock {\em Acta Mathematica}, 148:47--157, 1982.

\bibitem{HeMazzeo2017}
Siqi He and Rafe Mazzeo.
\newblock The extended {B}ogomolny equations and generalized {N}ahm pole
  boundary condition.
\newblock {\em Geom. Topol.}, 23(5):2475--2517, 2019.

\bibitem{hitchin1983construction}
N.~J. Hitchin.
\newblock On the construction of monopoles.
\newblock {\em Comm. Math. Phys.}, 89(2):145--190, 1983.

\bibitem{hurwitz1922}
A.~Hurwitz.
\newblock \"{U}ber die {K}omposition der quadratischen {F}ormen.
\newblock {\em Math. Ann.}, 88(1-2):1--25, 1922.

\bibitem{jardim2004surveynahm}
Marcos Jardim.
\newblock A survey on {N}ahm transform.
\newblock {\em J. Geom. Phys.}, 52(3):313--327, 2004.

\bibitem{kronheimer2004hyperkahler}
PB~Kronheimer.
\newblock A hyperkahler structure on the cotangent bundle of a complex lie
  group.
\newblock {\em arXiv preprint math/0409253}, 2004.

\bibitem{nagy2019complex}
{\'A}kos Nagy and Gon{\c{c}}alo Oliveira.
\newblock The {H}aydys monopole equation.
\newblock {\em Selecta Mathematica}, To appear.

\bibitem{nahm1980simple}
Werner Nahm.
\newblock A simple formalism for the {B}{P}{S} monopole.
\newblock {\em Phys. Lett., B}, 90(4):413--414, 1980.

\bibitem{nakajima1993monopoles}
Hiraku Nakajima.
\newblock Monopoles and {N}ahm's equations.
\newblock In {\em Einstein metrics and {Y}ang-{M}ills connections ({S}anda,
  1990)}, volume 145 of {\em Lecture Notes in Pure and Appl. Math.}, pages
  193--211. Dekker, New York, 1993.

\bibitem{salamonwalpuski2017}
Dietmar~A. Salamon and Thomas Walpuski.
\newblock Notes on the octonions.
\newblock In {\em Proceedings of the {G}\"{o}kova {G}eometry-{T}opology
  {C}onference 2016}, pages 1--85. G\"{o}kova Geometry/Topology Conference
  (GGT), G\"{o}kova, 2017.

\bibitem{uhlenbeck1986existence}
Karen Uhlenbeck and S.-T. Yau.
\newblock On the existence of {H}ermitian-{Y}ang-{M}ills connections in stable
  vector bundles.
\newblock {\em Comm. Pure Appl. Math.}, 39(S, suppl.):S257--S293, 1986.
\newblock Frontiers of the mathematical sciences: 1985 (New York, 1985).

\bibitem{witten2011fivebranes}
Edward Witten.
\newblock Fivebranes and knots.
\newblock {\em Quantum Topol.}, 3(1):1--137, 2012.

\end{thebibliography}
\end{document}